\documentclass[11pt,reqno]{amsart}
\usepackage{amssymb,xfrac}
\usepackage{amsmath}
\usepackage{color}
\usepackage{graphics, cite, bookmark}
\usepackage{epsfig}
\usepackage{hyperref}
\usepackage{epstopdf}
\usepackage{ esint }
\usepackage{amsthm}
\usepackage{lmodern}

\newtheorem{theorem}{Theorem}[section]
\newtheorem*{theorem*}{Theorem}
\newtheorem{lemma}[theorem]{Lemma}

\newtheorem{proposition}[theorem]{Proposition}

\theoremstyle{definition}

\theoremstyle{remark}
\newtheorem{remark}{Remark}[section]

\numberwithin{equation}{section}
\allowdisplaybreaks

\newcommand{\R}{\ensuremath{\mathbb{R}}}

\newcommand{\N}{\ensuremath{\mathbb{N}}}

\renewcommand{\S}{\ensuremath{\mathbb{S}}}

\newcommand{\sym}{\ensuremath{\mathrm{sym}}}
\newcommand{\Sym}{\ensuremath{\mathrm{Sym}_n}}

\usepackage{bbm}

\begin{document}

\title
[Nash-Kuiper theorem beyond Borisov]{A Nash-Kuiper theorem for isometric immersions beyond Borisov's exponent}

\author{Wentao Cao, Jonas Hirsch, and Dominik Inauen}

\address{Wentao Cao, Academy for Multidisciplinary Studies, Capital Normal University, West 3rd Ring North Road 105, Beijing, 100048 P.R. China. E-mail:{\tt cwtmath@cnu.edu.cn}}
\address{Jonas Hirsch, Institut f\"{u}r Mathematik, Universit\"{a}t Leipzig, D-04109, Leipzig, Germany.  E-mail:{\tt jonas.hirsch@math.uni-leipzig.de}}
\address{Dominik Inauen, Institut f\"{u}r Mathematik, Universit\"{a}t Leipzig, D-04109, Leipzig, Germany.  E-mail:{\tt dominik.inauen@math.uni-leipzig.de}}

\begin{abstract}

Given any short immersion from an $n$-dimensional bounded and simply connected domain into $\R^{n+1}$ and any H\"older exponent $\alpha<(1+n^2-n)^{-1}$, we construct a $C^{1, \alpha}$ isometric immersion arbitrarily close in the $C^0$ topology. This extends the classical Nash--Kuiper theorem and shows the flexibility of  $C^{1, \alpha}$ isometric immersions beyond Borisov's exponent. In particular, for $n=2$, the regularity threshold aligns with the Onsager exponent $1/3$ for the incompressible Euler equations. Our proof relies on three novelties that allow for the cancellation of leading-order error terms in the convex integration scheme: a new corrugation ansatz, an integration by parts procedure, and an adapted algebraic decomposition of these errors. 
\end{abstract}


\date{\today}



\maketitle 
\section{Introduction}

The isometric immersion problem is a fundamental problem in differential geometry. It seeks an immersion $ u: (\mathcal{M}, g) \to \R^m $ from an $ n $-dimensional Riemannian manifold $ (\mathcal{M}, g) $ into $ m $-dimensional Euclidean space that preserves the length of any $ C^1 $ curve. This condition is equivalent to the equality of the induced metric $ u^\sharp e$ and the intrinsic metric $ g $. In local coordinates, this translates to the system of $ n_* = n(n+1)/2 $ nonlinear partial differential equations:
\begin{equation}\label{e:iso}
g_{ij} = \partial_i u \cdot \partial_j u, \quad 1 \leq i, j \leq n,
\end{equation}
in $ m $ unknowns.

The classical Nash--Kuiper theorem \cite{Nash54,Kuiper} establishes that for $ m \geq n+1 $, any short immersion (or embedding) $ u: \mathcal{M} \to \R^m $ can be uniformly approximated by $ C^1 $ isometric immersions (resp. embeddings). Here,  an immersion $u$ is called \emph{short} if $\partial_i u \cdot \partial_j u \leq g_{ij}$ as quadratic forms. In particular, if the manifold is compact, any immersion can be made short by a homothety. The Nash--Kuiper theorem therefore demonstrates that when there are no topological obstructions to immersing the manifold into $ \R^m $ (a condition satisfied for $ m \geq 2n-1 $), there exists an abundance of $ C^1 $ solutions to \eqref{e:iso}. This abundance, often termed the \emph{flexibility} of isometric immersions,  is especially striking given the overdetermined nature of \eqref{e:iso} for $ m \geq n+1 $ and large $ n $. 

In contrast, classical \emph{rigidity} results show that \emph{smooth} isometric embedding into Euclidean spaces with such low codimension are typically unique. A prominent example is the rigidity theorem for the Weyl problem due to Cohn-Vossen \cite{CohnVossen1927} and Herglotz \cite{Herg1943}, which states that a $C^2$ isometric embedding $u:(\S^2, g) \to \R^3$, where $ g $ has  positive Gaussian curvature, is unique up to rigid motions.

 This strong contrast between flexibility in $C^1$ and rigidity in $C^2$ raises the natural question: is there a critical H\"older regularity threshold $ \alpha_0 $ that separates flexibility from rigidity? That is, does there exist a threshold $\alpha_0$ such that 
 \begin{itemize}
 \item if $\alpha>\alpha_0$, isometric immersions $u\in C^{1,\alpha}$ exhibit (some form of) rigidity;
 \item if $ \alpha < \alpha_0$, the Nash--Kuiper theorem extends to $C^{1,\alpha}$?
 \end{itemize}
 The precise value of $\alpha_0$ remains unresolved. Similarities between the iteration processes used to construct both isometric immersions or embeddings of regularity $C^{1,\alpha}$ and H\"older continuous weak solutions to the incompressible Euler equations (see e.g. \cite{CShighdim}) suggest the Onsager exponent $\alpha_0 = 1/3$ as a potential threshold. Indeed, the Onsager conjecture, which describes a similar phenomenon, states that for $C^{\alpha}$ weak solutions to the incompressible Euler equations
 \begin{itemize}
 \item[(1)]if $\alpha> 1/3$, it preserves the energy;
 \item[(2)]if $\alpha < 1/3$, the energy identity might be violated.
 \end{itemize}
 The rigidity result (1) was established in \cite{CET1994}, while the flexibility result (2) was ultimately resolved in \cite{Isett2018} following a sequence of works \cite{Desz2012, DeSz2014, BDIS2015, BDSV2019} that built on the groundbreaking approach of \cite{DeSz2009}, where the authors introduced a Nash-type iteration scheme to construct continuous weak solutions violating the energy identity.
 
 On the other hand, \cite{DI2020,CaoIn2024} demonstrate that for isometric embeddings, $C^{1,\sfrac{1}{2}}$ is a critical space \emph{in a suitable sense}, suggesting $\alpha_0 = 1/2$ (see also \cite[Question 36]{Gromov2017}, \cite[Section 10]{CShighdim}). 

This paper focuses on the flexibility side of this dichotomy for general dimension $n$ and codimension one. The case of general codimension $m\geq n+1$ is investigated in forthcoming work.

\subsection{Flexibility of $C^{1,\alpha}$ isometries: known results}

The study of $ C^{1,\alpha} $ isometric immersions dates back to the pioneering work of Yu. F. Borisov in the 1950s. Building on results of Pogorelov, Borisov proved in \cite{Borisov1958} that the Cohn--Vossen--Herglotz rigidity theorem extends to immersions of class $ C^{1,\alpha} $ when $ \alpha > 2/3 $ (see \cite{CDS} for an alternative proof).  On flexibility, Borisov announced in \cite{Borisov1965} that the Nash--Kuiper theorem extends to $ C^{1,\alpha} $ for
\begin{equation}\label{e:Borisovsexponent} \alpha < \frac{1}{1 + n^2 + n}\,, \end{equation} 
the \emph{Borisov exponent}, when $ \mathcal{M} $ is an $ n $-dimensional ball, with a potential improvement to $ \alpha < 1/5 $ for $ n = 2 $. He provided a proof for $ n = 2 $ and $ \alpha < 1/7 $ with an analytic metric in \cite{Borisov2004}. More recently, \cite{CDS} confirmed Borisov's claims, and in \cite{DIS} it was shown that for $ n = 2 $ one can indeed achieve the exponent  $ \alpha < 1/5 $.  These results were extended to general compact manifolds in \cite{CaoSze2022}.

\begin{remark}[High codimension]
If the codimension is large, more regular isometric immersions can be constructed. The breakthrough result in this direction is also due to Nash \cite{Nash56}, who proved that any $ (M, g) $ with $ g \in C^k $, $ k \geq 3 $, admits a $ C^k $-regular isometric immersion into $ \mathbb{R}^m $ for sufficiently large $ m $. Gromov \cite{GromovPdr} and G\"unther \cite{Guenther} improved the codimension bounds, the latter simplifying Nash's intricate iteration (now known as the \emph{hard implicit function theorem}). For less regular metrics $ g \in C^{l,\beta} $ with $ 0 < l + \beta < 2 $, K\"all\'en \cite{Kaellen} demonstrated that if $ m $ is large enough (one can show that $m\geq 2n_* + n $ suffices), there exists a $ C^{1,\alpha} $ immersion for any $ \alpha < (l + \beta)/2 $. See also \cite{DI2020,CaoIn2024,CaoSz2023} for further results in high codimension.
\end{remark}

\subsection{Statement of the result}

This paper improves the achievable H\"older exponent in the codimension one setting. Our main result is:

\begin{theorem}\label{t:main}
Let $ \Omega \subset \R^n $ be  any smooth bounded and simply connected domain  and $ g \in C^2(\overline{\Omega}, \Sym^+) $ a Riemannian metric. For any short immersion $ \underline{u} \in C^1(\overline{\Omega}, \R^{n+1}) $, any $ \epsilon > 0 $, and any
\begin{equation}\label{e:exponent}
\alpha < \frac{1}{1+n^2 - n},
\end{equation}
there exists an immersion $ u \in C^{1,\alpha}(\overline{\Omega}, \R^{n+1}) $ such that
\[ Du^tDu = g \quad \text{and} \quad \|u - \underline{u}\|_0 < \epsilon. \]
\end{theorem}

\begin{remark}
For $ n = 2 $, Theorem \ref{t:main} yields flexibility of $ C^{1,\sfrac13-} $ isometric immersions from surfaces to $ \R^3 $, aligning with the H\"older exponent of Onsager's conjecture. Given the parallels between convex integration solutions for Onsager's conjecture and isometric immersions (see e.g. \cite{BuckmasterVicol,CShighdim}) and the heuristic discussion about the H\"older exponent $\alpha<(1+2N)^{-1}(N\geq1)$ in Section \ref{s:borisov},  $ 1/3 $ may be optimal for codimension-one flexibility.
\end{remark}

\begin{remark}
    Following classical arguments (see e.g. \cite[Section 8]{DIS}), one can show that if the short map $\underline u$ in Theorem \ref{t:main} is an embedding, then $u$ can be chosen to be an embedding as well. 
\end{remark}

\begin{remark}
With additional effort, the key iteration Proposition \ref{p:stage} can be adapted to the framework of  \cite{CaoSze2022}, leading to a global version of Theorem \ref{t:main}. We prioritize the local version here for clarity. Moreover, our approach is easily adaptable to prove analogous results for very weak solutions to the Monge--Amp\`ere equation and system (see \cite{LewickaPakzad, CaoSz, CHI, MartaSystems,IL25,Martasystems2,Martasystems3}).
\end{remark}

\subsection{Main ideas}\label{s:mainideas}  In this subsection, we briefly recall the classical construction procedure--Nash's iteration for isometric immersions--and provide a heuristic explanation of how Borisov's exponent \eqref{e:Borisovsexponent} is obtained in \cite{CDS}. We then outline the main strategy for proving Theorem \ref{t:main} and highlight how our approach differs from previous methods. To focus on the core ideas, we adopt the local setting of Theorem \ref{t:main}: the manifold $ \mathcal M $ is described by a single coordinate chart $\Omega$, the Riemannian metric $ g $ is a matrix-valued function, and the induced metric $ u^\sharp e $ is given by the matrix field $ D u^t D u $.
 
\subsubsection{Nash's iteration}  

Following Nash \cite{Nash54}, the isometric immersion is obtained as the limit of an iteratively constructed sequence $\{u_q\}$ of strictly short immersions, whose induced metric converges to the intrinsic metric $g$ while $\{u_q\}$ remains Cauchy in some $C^{1,\alpha}$ space.  

The construction of $u_{q+1}$ from  $u_q$, referred to as a \emph{stage} (same terminology as in \cite{Nash54} and subsequent works), consists of a finite number of \emph{steps} designed to correct the metric deficit  $g - Du_q^tDu_q$.
This deficit is first decomposed into a finite sum of \emph{primitive metrics}, i.e., rank-one tensors with positive coefficients:  
\begin{equation}\label{e:decompintro}
g -  Du_q^tDu_q = \sum_{i=1}^N a_i^2 \nu_i \otimes \nu_i,
\end{equation}
where the directions $\nu_i \in \mathbb{S}^{n-1}$ are constant and the coefficients $a_i$ are smooth functions.  
After this decomposition, the short immersion $u_q$ is perturbed by $N$ \emph{steps} as follows: 

Setting $u_{q,0} = u_q$, one defines iteratively
\[
u_{q,i} = u_{q,i-1} + W_{q+1,i}, \text{ for } i = 1, \dots, N\,,\quad  u_{q+1}= u_{q,N}\,.
\]
Each function  $W_{q+1,i}$ is a highly oscillatory perturbation, and is chosen so that the corresponding update increases the induced metric approximately by $a_i^2 \nu_i \otimes \nu_i$, yielding  
\[
Du_{q,i}^tDu_{q, i} = Du_{q,i-1}^tDu_{q, i-1} + a_i^2 \nu_i \otimes \nu_i + E_i,
\]
where the error term $E_i$ can be made arbitrarily small by selecting a sufficiently high oscillation frequency $\lambda_i$ of $W_{q+1,i}$.  
Heuristically, the ansatz for $W_{q+1,i}$ has the form 
\begin{equation}\label{e:ansatzintro} W_{q+1,i} = \frac{a_i}{\lambda_i}\left(\gamma_1(\lambda_i x\cdot \nu_i) t_i +  \gamma_2(\lambda_i x \cdot \nu_i) \zeta_i\right)\end{equation} 
where the frequency $\lambda_i\gg 1 $, $t_i$ is a suitable \emph{tangent vector} (or normal to $ u_{q,i-1}$ in the case $m\geq n+2$ as in \cite{Nash54}), $\zeta_i$ is a unit normal vector to $u_{q,i-1}$ and $\gamma_1,\gamma_2$ are suitable periodic functions.\footnote{In fact, the corrugations used in \cite{Kuiper} and \cite{CDS} have the more complicated expression \[
 W_{q+1,i} = \frac{1}{\lambda_i}\left( \Gamma_1(a_i,\lambda_i x\cdot \nu_i)t_i+ \Gamma_2(a_i, \lambda_i x\cdot \nu_i)\zeta_i\right)\] 
 for some  suitable functions $\Gamma_1,\Gamma_2$ periodic in the second component, see also \cite{Laszlolecturenotes}.} 

By choosing the oscillation frequency large enough, one can ensure a geometric decay of the metric deficit:  
\[
\|g - Du_{q+1}^tDu_{q+1} \|_0 \leq \frac{1}{K} \|g - Du_q^tDu_q\|_0.
\]
Meanwhile, the step-wise corrections yield a bound on the $C^1$-norm growth:  
\begin{equation}\label{e:C1intro}
\|u_{q+1} - u_q\|_1 \leq C \sum_{i=1}^N \|a_i\|_0 \leq C \|g - Du_q^tDu_q\|_0^{\sfrac{1}{2}}
\end{equation}
where the last estimate follows from \eqref{e:decompintro}.
Given the geometric decay in $\|g - Du_q^tDu_q\|_0$, this ensures that the sequence $u_q$ remains Cauchy in $C^1$, converging to a limiting function that is an isometric immersion.  

\subsubsection{Borisov’s exponent and the result of \cite{CDS}}\label{s:borisov}

To achieve convergence in $C^{1,\alpha}$, a refined choice of frequencies is required to control the blow-up of the sequence $\{\|u_q\|_2\}$. The error term $E_i$ in the $i$-th step can be shown to satisfy  
\[
\|E_i\|_0 \leq C\|g - Du_q^tDu_q\|_0 \frac{\lambda_{i-1}}{\lambda_i}.
\]
Setting $\lambda_i = K\lambda_{i-1}$, we achieve the geometric decay  
\[
\|g - Du_{q+1}^tDu_{q+1}\|_0 \leq C K^{-1} \|g - Du_q^tDu_q\|_0,
\]
but at the cost of a $C^2$-norm growth of order $K^N$, where $N$ is the number of primitive metrics required in the decomposition  \eqref{e:decompintro}. This leads to  
\[
\|u_{q+1} - u_q\|_2 \leq C(K^N)^q\|u_1 - u_0\|_2.
\] Using \eqref{e:C1intro} and interpolation estimates, the sequence converges in $C^{1,\alpha}$ for any
\[
\alpha < \frac{1}{1+2N}.
\]
If the metric deficit $g - Du_q^tDu_q$ is close to a constant positive definite matrix (which can be achieved by a rescaling), it can be decomposed in exactly $N=  n(n+1)/2$ primitive metrics. This provides a heuristic explanation of Borisov's exponent \eqref{e:Borisovsexponent} and captures the core idea of the proof in \cite{CDS}. A technical challenge in turning  this idea into a rigorous proof lies in the loss of derivatives appearing in the above estimates. In \cite{CDS} and subsequent works, this loss is managed through an additional mollification step at the beginning of each stage.

\subsubsection{Prior improvements}

Borisov's exponent \eqref{e:Borisovsexponent} has been improved  previously via the following three approaches. We will discuss and show below the obstacles applying them to  attain Theorem \ref{t:main}.

(1) \textit{Reducing the number $N$ in the decomposition \eqref{e:decompintro}.} In \cite{DIS}, the authors studied  the case $n=2$, where isothermal coordinates can be used to diagonalize the metric deficit $g-Du_q^tDu_q$. This reduces $N$ from 3 to 2, leading to the improved H\"older exponent  $\alpha < 1/5$. However, for $n\geq 3$, no analogous coordinate transformation is known. 

(2) \textit{Increasing codimension.} When the codimension $m-n$ is large, multiple perturbations can be applied simultaneously at the same frequency, thereby reducing the growth of the $C^2$ norm of $\{u_q\}$. A procedure analogous to that used in \cite{MartaSystems} in the context of the Monge-Amp\`ere system would yield the improved exponent 
\[
\alpha < \frac{1}{1+2\frac{n_*}{m-n}}.
\] 
But our setting is codimension one.

(3) \textit{Absorbing the error terms $E_i$ directly into the decomposition \eqref{e:decompintro}.}  For large codimension $m - n \geq 2n_*$, from the second approach, it seems that at most $\alpha<1/2$ can be obtained. However, K\"all\'en \cite{Kaellen} introduced a new method that incorporates error terms into the decomposition:  
\[
g- Du_q^tDu_q = \sum_{i=1}^N a_i^2 \nu_i \otimes \nu_i + \sum_{i=1}^N E_i.
\]
Since the error terms $E_i$ depend on $a_i$ and its derivatives, this absorption is nontrivial. However, by applying a suitable Picard iteration to this equation, the decomposition can be achieved up to an arbitrarily small error 
\[\|E\|_0 \leq C \|g- Du_q^tDu_q\|_0\left(\frac{\lambda_0}{\lambda_1}\right)^J.\] 
Consequently, by choosing a very slow increase in frequency, $\lambda_1 = K^{1/J} \lambda_0$, the exponent improves to 
\[
\alpha < \frac{1}{1+\frac{2}{J}}.
\]
By taking $J$ arbitrarily large, one can achieve any $\alpha < 1$.  

When codimension $m-n\geq1$, absorbing error terms into the decomposition does not work, since the perturbations take the form of corrugations rather than Nash's spirals as in \cite{Kaellen}, and the structure of the error terms prevents the Picard iteration procedure from working. Nevertheless, in the context of the 2-dimensional Monge-Amp\`ere equation (corresponding to $n=2, m=3$), \cite{CHI} observed that the Picard scheme still works if applied only to $E_1$ with its 22 component removed. This modification allows for the trade of a fast derivative for a slow one, making the scheme work. The missing 22 component is then cancelled exactly by the second perturbation. 
To apply the analogous strategy to codimension one isometric immersions is significantly more challenging due to the nonlinear nature of the required equation. Furthermore, although the approach seems to generalize  to higher dimensions, the resulting improvement on the H\"older exponent is minor.

\subsubsection{Our approach--iterative integration by parts}\label{s:ibp} 

In this paper, we improve the H\"older exponent in the codimension one setting \emph{without} reducing the number of primitive metrics in the decomposition of the metric deficit and (almost) without absorbing error terms into the decomposition. Instead, we modify the ansatz \eqref{e:ansatzintro} for the corrugation in the first $ n $ perturbations, achieving an error with the very small bound:  
\[
\|E_i\|_0 \leq C\|g - Du_q^tDu_q\|_0 \left(\frac{\lambda_{i-1}}{\lambda_i}\right)^J
\]
up to a component that belongs to an $(n_* - n)$-dimensional subspace of the space of symmetric matrices. This remaining part is then cancelled exactly in the last $ n_* - n $ perturbations. By choosing $ J $ arbitrarily large, the contribution of the first $ n $ perturbations to the growth of the $C^2$ norm becomes negligible, effectively reducing the number of contributing perturbations to $ n_* - n $. This explains the H\"older threshold claimed in Theorem \ref{t:main}.

The key observation behind modifying the ansatz \eqref{e:ansatzintro} is that, at each step $ i $, all but one of the leading-order error terms in $E_i$ take the form  
\[
\gamma(\lambda_i x \cdot \nu_i) M,
\]  
where $ M $ is a symmetric matrix oscillating at the lower frequency $ \lambda_{i-1} $, and $ \gamma $ is a periodic function with zero mean.  
Another observation is that for a fixed $ \nu \in \mathbb{S}^{n-1} $, any symmetric matrix can be decomposed as  
\[
M = \sym (\alpha(M) \otimes \nu) + F,
\]
where the ``remainder" $ F $ belongs to an $(n_* - n)$-dimensional subspace of the space of symmetric matrices. Using this decomposition, we can rewrite  
\[
\gamma(\lambda_i x\cdot \nu_i) M = 2\sym\left( D\left( \frac{\gamma^{(1)}(\lambda_i x\cdot\nu_i)}{\lambda_i} \alpha(M)\right)\right) - 2\frac{\gamma^{(1)}}{\lambda_i} \sym(D\alpha(M)) + \gamma  F.
\]  
Here, $ \gamma^{(1)} $ is an antiderivative of $ \gamma $ with zero mean. We refer to this process as \emph{integration by parts}.  
If $M$ and $\gamma$ are sufficiently smooth, we can iterate this process $ J $ times, and construct fields $ w^J, E^J, F^J $ such that  
\[
\gamma M = 2\sym (D w^J) + \left(\frac{\lambda_{i-1}}{\lambda_i}\right)^J E^J+ F^J,
\]  
where the remainder $ F^J $ remains in a lower-dimensional subspace (see Proposition \ref{p:ibp}). Consequently, introducing vector field $w^J$ in the step ansatz so that $ 2\,\sym(D w^J) $ appears in the induced metric allows us to cancel $\gamma M$ up to a very small error $ \left(\frac{\lambda_{i-1}}{\lambda_i}\right)^J E^J $ and the large but lower-dimensional remainder $ F^J $.  Furthermore, by selecting a suitable basis $ \{\nu_i \otimes \nu_i\} $ for the decomposition \eqref{e:decompintro} we ensure that, when applied to the first $n$ perturbations (i.e., for $ \nu = \nu_i $, $ i = 1, \dots, n $), the integration by parts procedure produces errors $ F^J $ that all remain within the same subspace $\mathcal{V}= \mathrm{span} \{\nu_j \otimes \nu_j : j = n+1, \dots, n_*\} $. These errors in $\mathcal{V}$ can then be canceled exactly by appropriately adjusting the effective amplitude $ a_j $ of the perturbation $ W_{q,j} $ for $ j \geq n+1 $.


A complication arises due to a specific leading-order error term in $E_i$ of the form  
\[
\gamma^2 M,
\]  
which prevents a direct application of integration by parts technique since $ \gamma^2 $ does not have zero mean. However, we can decompose  
\[
\gamma^2 M = (\gamma^2 - \fint \gamma^2)M + \fint \gamma^2 M.
\]  
The prefactor of the first term has zero mean, allowing the first term to be handled using the integration by parts process described above. On the other hand, the remaining term $ \fint \gamma^2 M $ oscillates slowly, so that can be treated by absorbing it into the decomposition via a simple Picard iteration, similar to K\"all\'en’s approach (see Lemma \ref{l:babykaellen}).

\subsection{Structure of the remaining part of the paper} 
In Section \ref{s:preliminaries} we gather decomposition lemmas and  a proposition for iterative integration by parts. 
We then introduce the new ansatz for the perturbation in Section \ref{s:iteration} (see \eqref{d:perturbation}). It differs from the perturbations used in previous works by the introduction of $w$ into the tangential component of the perturbation,  a modification inspired by the closely related iteration process for the Monge-Amp\`ere equation. Additionally, the tangential component is smaller than usual  (observe the $\delta$ instead of the usual $\delta^{\sfrac{1}{2}}$ in front of it). 
We then apply this general ansatz in the two distinct cases: when the direction $\nu_i$ corresponds to $i=1,\ldots,n$ and when $i\geq n+1$. In the first case, we use   
integration by parts to obtain a suitably small error, as shown in Proposition  \ref{p:ibpstep}. In the second case we simply set $w=0$ and recover the usual stepwise bounds (Proposition \ref{p:usualstep}). We then establish the main iterative result, Proposition \ref{p:stage}. 
Finally, we complete the proof of Theorem \ref{t:main} in Section \ref{s:pfoftmain} by iteratively applying Proposition \ref{p:stage}. 

\subsection{Notation} For $A\in \R^{n\times n} $ we set $\sym(A) = \frac{1}{2}(A+A^t)$, where $A^t$ is the transpose of $A$. Denote $\Sym$ as the set of all $n\times n$ symmetric matrices. For $\zeta,\xi \in \R^n$ we set $\zeta\odot \xi = \sym(\zeta\otimes \xi) $. $C$ represents a positive constant which may differ from line to line, and its value may rely on the parameters specified.

\section{Preliminaries}\label{s:preliminaries}
\subsection{Decomposition lemmas.}
As discussed in the introduction, we want to decompose the metric deficit at the beginning of each stage as in \eqref{e:decompintro}. Because of the integration by parts procedure, we first fix a favorable basis $\{\nu_i\otimes\nu_i\}$ of $\Sym$. 

Let $e_i\in \R^n$ denote the $i$-th standard basis vector. We choose the  basis $\{\nu_i\otimes \nu_i : i=1,\ldots, n_*\}$ of $\Sym$, where 
\[\{\nu_i:i=1,\ldots, n_*\}=\left\{\frac{e_i+e_j}{|e_i+e_j|}, 1\leq i,j\leq n\right\}.\]
Moreover, we order $\{\nu_i\}$ such that $\nu_j \cdot e_1 \neq 0 $ for $j=1,\ldots, n$. Fix the positive definite matrix
\begin{equation}\label{d:h_0}
h_0 =\sum_{j=1}^{n_*} \nu_j\otimes \nu_j\,.\end{equation}

We then have the following decomposition lemma (compare also \cite[Lemma 5.2]{CDS}).
\begin{lemma}\label{l:matrixdecomposition} 
There exist $r_1=r_1(n)>0$ and linear maps $L_j: \Sym \to \R$ such that 
\[ h= \sum_{j=1}^{n_*} L_j(h) \nu_j\otimes \nu_j \]
for all $h\in \Sym$, and moreover $L_j(h)\geq r_1$ if $|h-h_0|\leq r_1$.
\end{lemma}
\begin{proof}
    The existence of $L_j$ follows from the fact that $\{\nu_j\otimes \nu_j\}$ is a basis for $\Sym$. The existence of $r_1$ follows by continuity and the fact that $L_j(h_0) = 1 $ for all $j$. 
\end{proof}

As discussed in Section \ref{s:ibp}, not all leading order error terms are amenable to integration by parts. Nevertheless, we can absorb the remainder into the decomposition with arbitrarily high precision through the following K\"all\'en-type decomposition. Notice  that there is no rapidly oscillating factor in front of the error term to be absorbed, simplifying the proof. 

\begin{lemma}\label{l:babykaellen} Let $N\in \N$. There exists $r_2=r_2(n, N)>0$ such that the following holds for any $1\leq \lambda_0\leq \lambda_1\leq \cdots \leq \lambda_n$ and $H\in C^{N+1}(\bar \Omega,\Sym)$ satisfying
\begin{align}
    &\|H-h_0\|_0 + \frac{\lambda_0}{\lambda_1}< r_2\,,  \label{a:H}\\ 
    &\|H\|_k \leq \lambda_0^k \text{ for } k=1,\ldots, N+1.\label{a:Hestimates}
\end{align}
    For any $j=0,\ldots, N$ there exists a vector $a^j =(a^j_1,\ldots, a^j_{n_*}) \in C^{N+1-j}(\bar \Omega, \R^{n_*})$ and $E^j\in C^{N-j}(\bar \Omega, \Sym)$ such that 
    \begin{equation}\label{e:Kaellendecomp}
        H= \sum_{i=1}^{n_*} (a^j_i)^2 \nu_i\otimes \nu_i + \sum_{l=1}^n \frac{2\pi}{\lambda_l^2}\nabla a^j_l\otimes \nabla a^j_l + E^j\,,
    \end{equation}
    and the estimates 
     \begin{align}
         &a^j_i\geq r_2, \text{ for }i=1,\ldots, n_*\,,\label{e:ajlower}\\
         &\|a^j\|_k \leq C \lambda_0^k \text{ for }k=0,\ldots, N-j+1\,,\label{e:ajestimate}\\
         & \|E^j\|_k \leq C \left(\frac{\lambda_0}{\lambda_1}\right)^{2(j+1)}\lambda_0^k \text{ for } k=0,\ldots,N-j \,,\label{e:ejestimate}
     \end{align}
     hold. The constant $C$ only depends on $n, N$. 
\end{lemma}
\begin{proof}
We prove it by induction. Let $j=0$. If $r_2<r_1$ we can apply Lemma \ref{l:matrixdecomposition} to write 
\[ H= \sum_{i=1}^{n_*} (a_i^0)^2\nu_i\otimes \nu_i =\sum_{i=1}^{n_*} (a^0_i)^2 \nu_i\otimes \nu_i + \sum_{l=1}^n \frac{2\pi}{\lambda_l^2}\nabla a^0_l\otimes \nabla a^0_l + E^0 \]
for 
\[ a_i^0 = \sqrt{ L_i(H)}\,,\, E^0 = -\sum_{l=1}^n \frac{2\pi}{\lambda_l^2}\nabla a^0_l\otimes \nabla a^0_l \,.\]
Note that $a_i^0$ is well-defined and $C^{N+1}$ due to $L_i(H) \geq r_1 >0$ as a consequence of Lemma \ref{l:matrixdecomposition}, and that \eqref{e:ajlower} holds if $r_2\leq \sqrt{r_1}$. It is easy to derive from the assumption \eqref{a:H} that 
\[\|H\|_0\leq \|h_0\|_0+\|H-h_0\|_0\leq C,\]
where $C$ depends only on $n$ and $N$. Along with assumption \eqref{a:Hestimates}, estimate \eqref{e:ajestimate} then follows from Lemma \ref{l:composition}, while \eqref{e:ejestimate} follows from the just established \eqref{e:ajestimate} and the Leibniz rule \eqref{e:Leibniz}.

 Assume now that Lemma \ref{l:babykaellen} holds until some $j\leq N-1$. We will show it also holds for $j+1$. In fact, with $a^j$ already defined, by \eqref{e:ajestimate}, we get
\[\|p^j\|_0\leq C\frac{\lambda^2_0}{\lambda_1^2},\]
where we denoted \[ p^j=\sum_{l=1}^n\frac{2\pi}{\lambda_l^2}\nabla a^j_l\otimes \nabla a^j_l.\]
Together with assumption \eqref{a:H} this implies
\[\|H-p^j-h_0\|_0\leq \|H-h_0\|_0+C\frac{\lambda^2_0}{\lambda_1^2}\leq r_2+Cr_2^2\leq r_1,\]
provided that $r_2$ is small enough depending only on $N$ and $n$. By Lemma \ref{l:matrixdecomposition}, we can therefore decompose 
\[H-p^j=\sum_{i=1}^{n_*}(a_i^{j+1})^2\nu_i\otimes\nu_i,\]
for \[\displaystyle  a_i^{j+1}=\sqrt{L_i\left(H-p^j\right)}\,.\]
Again, this is well defined by Lemma \ref{l:matrixdecomposition}, and \eqref{e:ajlower} holds if $r_2\leq \sqrt{r_1}$.  Moreover, the estimates \eqref{e:ajestimate} follow from Lemma \ref{l:composition}, using the induction assumption and assumption \eqref{a:Hestimates}. On the other hand, setting
\[E^{j+1}=\sum_{l=1}^n\frac{2\pi}{\lambda_l^2}(\nabla a^j_l\otimes \nabla a^j_l-\nabla a^{j+1}_l\otimes \nabla a^{j+1}_l)\,\]
 yields \eqref{e:Kaellendecomp}.
Regarding \eqref{e:ejestimate} we can estimate, using \eqref{e:ajestimate},
\begin{align*}
    \|E^{j+1}\|_k\leq &\sum_{l=1}^n\frac{2\pi}{\lambda_l^2}\|\nabla a^j_l\otimes \nabla a^j_l-\nabla a^{j+1}_l\otimes \nabla a^{j+1}_l\|_k\\
    \leq& \frac{C}{\lambda_1^2}\sum_{l=1}^n\left(\|\nabla a_l^j\otimes(\nabla a_l^j-\nabla a_l^{j+1})\|_k+\|(\nabla a_l^j-\nabla a_l^{j+1})\otimes \nabla a_{l}^{j+1}\|_k\right)\\
    \leq& \frac{C}{\lambda_1^2}\sum_{l=1}^n\left(\lambda_0\|(\nabla a_l^j-\nabla a_l^{j+1})\|_k+\lambda_0^{k+1}\|(\nabla a_l^j-\nabla a_l^{j+1})\|_0\right)\,.
\end{align*}
From \eqref{e:Kaellendecomp} and the fact that $L$ is linear it follows 
\[ a^j_l = \sqrt{L_l(H-p^j-E^j)}\,,\]
so that
\begin{align*}
    a_l^{j+1}-a_l^j
    =&\sqrt{L_l\left(H-p^j\right)}-
    \sqrt{L_l\left(H-p^{j}-E^j\right)}\\
     =& \frac{L_l(E^j)}{\sqrt{L_l\left(H-p^j\right)}+\sqrt{L_l\left(H-p^{j}-E^j\right)}}\,.
\end{align*}
By the Leibniz rule, Lemma \ref{l:composition} and estimates \eqref{e:ajlower}--\eqref{e:ejestimate} we obtain
\begin{align*}
     \|a_l^j-a_l^{j+1}\|_{k+1}\leq& C\Big(\big(\|H\|_{k+1}+\|p^{j}\|_{k+1}+\|E^j\|_{k+1}\big)\|E^j\|_0+\|E^j\|_{k+1}\Big)\\
     \leq&C\|E^j\|_0\big(\lambda_0^{k+1}+\|\nabla a_l^{j-1}\otimes\nabla a^{j-1}_l\|_{k+1}\lambda_1^{-2}\big)\\
     &\quad +C\left(\frac{\lambda_0}{\lambda_1}\right)^{2(j+1)}\lambda_0^{k+1}\|E^j\|_0+C\|E^j\|_{k+1}\\
\leq&C\left((\lambda_0^{k+1}+\lambda_0^{k+3}\lambda_1^{-2})+\lambda_0^{k+1}\right)\left(\frac{\lambda_0}{\lambda_1}\right)^{2(j+1)}\\
     \leq& C\lambda_0^{k+1}\left(\frac{\lambda_0}{\lambda_1}\right)^{2(j+1)}\,
\end{align*}
for all $k=0,\ldots,N-j-1$. Combining with the above yields
\[ \|E^{j+1}\|_k\leq \frac{C}{\lambda_1^2}\lambda_0^{k+2}\left(\frac{\lambda_0}{\lambda_1}\right)^{2(j+1)}\,,\]
for $k=0,\ldots, N-j-1$,
which gives \eqref{e:ejestimate} for $j+1$ case. This completes the proof.
\end{proof}

\subsection{Iterative integration by parts} We now introduce the key integration by parts procedure in Proposition \ref{p:ibp}. Before doing so we need the following algebraic decomposition lemma.
\begin{lemma}\label{l:ibpdecomposition}
    Let $\nu\in \S^{n-1}$ such that $\nu\cdot e_1 \neq 0$. Then the linear map $\Phi:\R^n \times \R^{n_*-n}\to \Sym$ defined by 
    \[ \Phi(\alpha,\beta) = \alpha \odot \nu + \sum_{j=n+1}^{n_*} \beta_{j-n} \nu_j\otimes \nu_j \]
    is an isomorphism with inverse denoted by \[\Psi(M) := (\alpha(M),\beta(M)) := (\alpha(M),\beta_1(M),\ldots, \beta_{n_*-n}(M) ).\]
\end{lemma}

\begin{proof}
Due to the fact that $\Phi$ is a linear map from an $n_*$-dimensional vector space to another vector space of the same dimension, it is enough to show that $\Phi$ is injective. In the following we show the kernel of $\Phi$ is zero. In fact, assume $(\alpha, \beta)$ are such that
\[\Phi(\alpha, \beta)=\alpha \odot \nu + \sum_{j=n+1}^{n_*} \beta_{j-n} \nu_j\otimes \nu_j=0,.\]
Taking the dot product of both sides with $e_1$ we get
\[\alpha\odot\nu\cdot e_1=0,\]
which then implies 
\[\alpha=-\frac{\alpha\cdot e_1}{\nu\cdot e_1}\nu\]
by $\nu\cdot e_1\neq0.$ Thus we obtain
\[-\frac{\alpha\cdot e_1}{\nu\cdot e_1}\nu\otimes \nu + \sum_{j=n+1}^{n_*} \beta_{j-n} \nu_j\otimes \nu_j=0.\]
The linear independence of $\nu\otimes\nu$ and $\nu_j\otimes \nu_j, j=n+1, \ldots, n_*$ gives us 
$$\alpha\cdot e_1=0,\qquad \beta_j=0,\,\, j=1,\ldots, n_*-n,$$ 
which then implies
$(\alpha,\beta)=(0, 0).$ 
Therefore, the kernel of $\Phi$ is zero and then $\Phi$ is an isomorphism.
\end{proof}

The following proposition encapsulates the integration by parts process. The symmetric matrix $\gamma M$ represents a leading order error term within the error of after a step perturbation (which was denoted $E_i$ in Section \ref{s:mainideas}).

\begin{proposition}\label{p:ibp} 
Fix $N\in \N$, $\mu\geq 1 $, $K\geq 0$, and assume $M\in C^N(\overline \Omega, \Sym)$ satisfies
\begin{align} 
 &\|M\|_k \leq K \mu^k \,, \text{ for } k=0,\ldots, N\,. \label{assM}
 \end{align}
  Fix a vector $\nu\in \S^{n-1}$ with $\nu\cdot e_1 \neq 0$, a frequency $\lambda\geq \mu $,  and a smooth periodic function $\gamma\in C^\infty(\S^1)$ with zero mean. 
 
 Then for any natural number $i\in[1, N]$, there exist a vector field $w^i\in C^{N-i+1}(\overline \Omega, \R^n)$, matrix fields $E^i\in C^{N-i}(\overline \Omega, \Sym)$,  $F^i\in C^{N-i+1}(\overline \Omega,\Sym)$,  and a function $\gamma^i\in C^\infty(\S^1)$ with zero mean such that 
 \begin{equation}\label{e:ibpdecomposition}
 \gamma(\lambda x\cdot \nu ) M = 2\,\sym(D w^i) +\gamma^{i}(\lambda x\cdot \nu)\left(\frac{\mu}{\lambda}\right)^i E^i + F^i\,,
 \end{equation}
 and 
  \begin{equation}\label{e:spanproperty}
  F^i \in \mathrm{span}\{ \nu_j\otimes \nu_j: j=n+1,\ldots, n_*\}\,.\end{equation}
  Moreover, the following estimates are satisfied
\begin{align} 
  \|w^i\|_k &\leq C K \lambda^{k-1} \,,\text{ for } k=0,\ldots, N-i+1\,, \label{e:westimate} \\
 \|E^i\|_k &\leq C K \mu^{k}\,,\text{ for } k=0,\ldots,N-i\,,\label{e:Eestimate}\\
 \|F^i\|_k &\leq C K \lambda^k\,,\text{ for } k=0,\ldots, N-i+1\,,\label{e:Festimate}
\end{align}    
 where $C$ is a constant only depending on $N,\, n, \,\nu$ and $\gamma$.
 
\end{proposition}

\begin{proof}
    We prove the proposition by induction. For $i=1$ we use Lemma \ref{l:ibpdecomposition} to decompose 
    \[ M = \alpha(M)\odot \nu +\sum_{j=n+1}^{n_*} \beta_{j-n}(M)\nu_j\otimes \nu_j\,. \]
    Since $\Psi $ is linear we have  for all $k=0,\ldots, N$
    \[ \|\alpha(M)\|_k+\|\beta(M)\|_k\leq C \|M\|_k \]
    for a constant $C$ only depending on $n,N$ and $\nu$. Set 
    \[ \gamma^1(t) = \int_0^t \gamma(s)\,ds - \fint_0^{2\pi} \left( \int_0^t \gamma(s)\,ds\right )\,dt\,.\]
 By the assumption on $\gamma$, we immediately have $\gamma^1\in C^\infty(\S^1)$ and
 \[\fint \gamma^1 =0,\quad \frac{d}{dt} \gamma^1 = \gamma.\]
Define
\[w^1(x) = \frac{\gamma^1(\lambda x\cdot \nu)}{2\lambda} \alpha(M(x))\,.\]
Then $w^1\in C^N$ with  
\[ \|w^1\|_k \leq C K \lambda^{k-1} \text{ for } k=0,\ldots, N,\]
where $C$ only depends on $n,N$ and $\gamma$,
thanks to the Leibniz rule \eqref{e:Leibniz}, $\mu\leq \lambda$ and the assumptions on  $M$. Then \eqref{e:westimate} holds for $i=1$.
Observe that 
\[ Dw^1 = \frac12\gamma \alpha(M)\otimes \nu + \frac{\gamma^1}{2\lambda} D(\alpha(M))\,,\]
where we suppressed the argument $\lambda x\cdot \nu$ of $\gamma $ and $\gamma^1$. Consequently,
\[ 2\,\sym(D w^1) = \gamma \alpha(M)\odot \nu + \frac{\gamma^1}{\lambda} \sym\left(D(\alpha(M))\right)\,.\]
This yields 
\[ \gamma M = 2\,\sym ( Dw^1)-  \frac{\gamma^1}{\lambda} \sym\left( D(\alpha(M))\right) +\gamma \sum_{j=n+1}^{n_*} \beta_{j-n}(M)\nu_j\otimes \nu_j\,,\]
giving \eqref{e:ibpdecomposition} for $i=1$ after defining 
\[ E^1 = -\frac{1}{\mu} \sym\left(D(\alpha(M))\right)\,,\quad F^1 = \gamma \sum_{j=n+1}^{n_*} \beta_{j-n}(M)\nu_j\otimes \nu_j\,.\]
Clearly, \eqref{e:spanproperty} holds. The remaining estimates \eqref{e:Eestimate} and \eqref{e:Festimate} then follow from the assumption \eqref{assM}, $\lambda\geq \mu$ and the Leibniz rule \eqref{e:Leibniz}.

Now assume the claim is true up to some $1\leq i\leq N-1$, so that we can write
\[ \gamma(\lambda x\cdot \nu ) M = 2\,\sym(D w^i) +\gamma^{i}(\lambda x\cdot \nu)\left(\frac{\mu}{\lambda}\right)^i E^i + F^i\,\]
for some $w^i,E^i,F^i,\gamma^i$ satisfying \eqref{e:spanproperty}--\eqref{e:Festimate}.
Define 
\[\tilde \gamma = \gamma^i,\quad \tilde N = N-i,\quad \tilde K = C\left ( \frac{\mu}{\lambda}\right)^i K,\quad
 \tilde M =  \left ( \frac{\mu}{\lambda}\right)^i E^i\,,\]
where $C$ is the constant in \eqref{e:Eestimate}. By the induction assumption there exist $\tilde w^1, \tilde F^1\in C^{\tilde N}$, $\tilde E^1\in C^{\tilde N-1}$ and $\tilde \gamma^1 \in C^\infty(\S^1) $ with zero mean such that 
\[\tilde \gamma(\lambda x\cdot \nu) \tilde M = 2\,\sym( D\tilde w^1 ) + \tilde \gamma ^1(\lambda x\cdot \nu) \frac{\mu}{\lambda} \tilde E^1 + \tilde F^1\,\]
and \eqref{e:spanproperty}--\eqref{e:Festimate} are satisfied with $\tilde K$, $\tilde N$ and $i$ replaced by  $K$, $N$, and $1$ respectively. We then define 
\[
    w^{i+1} = w^i + \tilde w^1,\quad F^{i+1} = F^i + \tilde F^1,\quad  E^{i+1} = \left(\frac{\lambda}{\mu}\right)^i \tilde E^1,\quad    \gamma^{i+1}= \tilde \gamma^1\,,
\]
so that \eqref{e:ibpdecomposition} and \eqref{e:spanproperty} hold for $i+1$. For the estimates \eqref{e:westimate} we get from $\tilde K\leq C K$ and the induction assumption that
\[ \|w^{i+1}\|_k \leq CK\lambda^{k-1} + C\tilde K \lambda^{k-1}\leq CK\lambda^{k-1} \text{ for } k=0,\ldots, \tilde N = N-(i+1)+1\,.\]
The estimate \eqref{e:Festimate} follows analogously. Finally, 
\[\|E^{i+1}\|_k \leq \left(\frac{\lambda}{\mu}\right)^i \|\tilde E^1\|_k \leq C \left(\frac{\lambda}{\mu}\right)^i \tilde K \mu^k = CK \mu^k  \]
for $k=0,\ldots, \tilde N-1 = N-(i+1)$, finishing the proof. 
\end{proof}

\section{Iteration Propositions}\label{s:iteration}

In this section we prove the main iteration proposition, which describes how to construct the short map $u_{q+1} $ from the given $u_q$ -- the \emph{stage} Proposition \ref{p:stage}. Given the decomposition of the metric deficit $g-Du_q^tDu_q$ via Lemma \ref{l:babykaellen}, we perturb the map $u_q$ through $n_*$ steps. Our new ansatz for each such perturbation is given in Lemma \ref{l:step} (see \eqref{d:perturbation}). We then show in Proposition \ref{p:ibpstep} that  this ansatz, when combined with the iterative integration by parts process (Proposition \ref{p:ibp}), yields favorable estimates on the metric error. Meanwhile, Proposition \ref{p:usualstep} describes the effect of the ansatz without integration by parts. Finally, we prove Proposition \ref{p:stage} by applying Proposition \ref{p:ibpstep} $n$ times and Proposition \ref{p:usualstep}  $n_*-n$ times.

\subsection{Adding one primitive metric--A Step}
We now introduce the general form of our new ansatz for the perturbation and investigate its effect under natural assumptions on the involved quantities. The perturbation is based upon the periodic corrugation functions 
$\gamma_1,\gamma_2 \in C^\infty(\S^1)$ defined by 
\begin{equation}\label{d:corrugations} 
\gamma_1(t) = -\frac{1}{4}\sin(2t) \,, \gamma_2(t) = \sqrt{2} \sin(t)\,,
\end{equation}
which is similar to the ones in  \cite{LewickaPakzad}. Observe that both $\gamma_1$ and $\gamma_2$ have mean zero and satisfy the ``inclusion"
\begin{equation}\label{e:inclusion}
    2\gamma_1' +(\gamma_2')^2 = 1\,.
\end{equation}

\begin{lemma}[Perturbation]\label{l:step} 
Fix $N\in \N$, $0\leq j\leq N$. Let $\nu \in\S^{n-1}$, $u\in C^{N+2}(\bar \Omega,\R^m)$, $a\in C^{N+1}(\bar \Omega)$ and $w\in C^{N-j+1}(\bar \Omega, \R^n)$ be given such that 
\begin{align}
    &\frac{1}{\rho}\mathrm{Id}\leq Du^t Du \leq \rho \mathrm{Id}\,, \label{a:uinvertibility}\\
    &\|u\|_{k+1}\leq \delta^{\sfrac{1}{2}} \mu^{k} \,, \text{ for }  k=1,\ldots,N,\label{a:uestimates}\\
    &\|a\|_0\leq \rho, \label{a:a0estimate}\\ 
    &\|a\|_k\leq \mu^k \text{ for } k=1,\ldots,N+1\,,\label{a:aestimates}\\
    & \|w\|_k \leq C_0\lambda^{k-1} \text{ for } k=0,\ldots, N-j+1 \label{a:westimates}
\end{align} 
for some constants $\rho>1, \,0<\delta\leq 1$,\, $\lambda\geq\mu\geq \delta^{-\sfrac12}$ and constant $C_0\geq 1$. Define the perturbation 
\begin{equation}\label{d:perturbation}
        v(x)= u(x) +\delta T(x) \left( \frac{a^2(x)\gamma_1(\lambda x\cdot \nu )}{\lambda}\nu + w(x) \right) + \frac{\delta^{\sfrac{1}{2}}a(x) \gamma_2(\lambda x\cdot \nu )}{\lambda }\zeta(x)\,,
    \end{equation}
where 
\begin{align}
    &T= Du (Du^tDu)^{-1}\,,\label{d:T}\\
    &Du^t \zeta = 0\,,\quad |\zeta| =1 \,. \label{a:normal}
\end{align}
Then $v\in C^{N-j+1}(\bar \Omega,\R^n)$ satisfies
\begin{equation}\label{e:stepestimates}
\|v-u\|_k \leq C\delta^{\sfrac{1}{2}}\lambda^{k-1} \text{ for } k=0,\ldots, N-j+1\,,
\end{equation}
and 
    \begin{align}\label{e:steperror}
     \begin{split}
         D v^t Dv =& Du^t Du + \delta a^2 \nu \otimes \nu +R+
         \delta \frac{2\pi}{\lambda^2}\nabla a\otimes \nabla a \\
     &+2\delta \,\sym\left(  Dw\right)+\delta\sum_{i=1}^4\gamma_i(\lambda x\cdot \nu)M_i\,,
     \end{split}
\end{align}
where
\[\gamma_3:=\gamma_2\gamma_2',\quad \gamma_4:=\gamma_2^2-2\pi,\] 
and $M_i\in C^N(\bar \Omega, \Sym), i=1,\ldots, 4$ are independent of $w$,  $R\in C^{N-j}(\bar \Omega, \Sym)$, and the estimates 
\begin{align}
    &\|M_i\|_k\leq \bar C \frac{\mu}{\lambda}\mu^k\, \text{ for } k=0,\ldots, N\,,\label{e:steperrorestimatesM}\\
    &\|R\|_k\leq C\delta^{\sfrac{3}{2}}\lambda^k, \text{ for } k=0,\ldots, N-j\,,\label{e:steperrorestimatesR}
\end{align}
are satisfied.
The constants  $C\geq 1$ in \eqref{e:stepestimates} and \eqref{e:steperrorestimatesR} depends on $n, N$, $\rho$ and $C_0$, while the constant $\bar C$ in \eqref{e:steperrorestimatesM} only  depends  on $n, N$, $\rho$ but not on $C_0.$
\end{lemma}

\begin{remark}
Our ansatz in \eqref{d:perturbation} differs from the usual codimension-one ansatz (see \cite{CDS}, \cite{DIS}, etc.) in three ways: a different choice of oscillatory functions, a smaller tangential component of the perturbation (scaling as $\delta$ instead of $\delta^{\sfrac{1}{2}}$), and, most importantly, the introduction of a corrective vector field $w$ in the tangential perturbation.  

A consequence of our choice of oscillatory functions \eqref{d:corrugations} is that they satisfy the inclusion \eqref{e:inclusion} instead of the more complicated (1.26) in \cite{CDS}. This results in an additional error in the induced metric $Dv^t Dv$ of order $0$ in $\lambda$, which initially appears too large. However, due to the smaller size of the tangential perturbation, this error is in fact of order $\delta^{\sfrac{3}{2}}$, making it sufficiently small for our purposes.  
\end{remark}
\begin{proof}
Observe that both $T$ in \eqref{d:T} and $\zeta$ in \eqref{a:normal} are well-defined thanks to \eqref{a:uinvertibility}. Using Lemma \ref{l:composition}, the Leibniz rule \eqref{e:Leibniz} and \eqref{a:uinvertibility}--\eqref{a:uestimates} we conclude  
\begin{align}
\|T\|_k\leq& C\|u\|_{k+1}\leq C(1+\delta^{\sfrac{1}{2}}\mu^k),\label{e:Testimate}\\
\|\zeta\|_k\leq& C\|u\|_{k+1}\leq C(1+\delta^{\sfrac{1}{2}}\mu^k),\label{e:zetaestimates}
\end{align} 
for $k=0,\ldots, N$ and some constant $C$ which depends only on $n, N$ and $\rho$ (compare e.g. \cite[Lemma 3.5]{CaoIn2024}). We now show the estimates \eqref{e:stepestimates}. Using \eqref{e:Testimate}, \eqref{a:a0estimate} and \eqref{a:aestimates} we find
 \begin{align*}
     \|\delta T \frac{a^2 \gamma_1(\lambda x\cdot \nu )}{\lambda}\nu\|_k \leq& C\delta\left( \lambda^{k-1}\|a^2 T\nu\|_0 + \lambda^{-1} \|a^2 T\nu\|_k\right)
     \leq C\delta \lambda^{k-1}.
 \end{align*}
 for $k=0,\ldots, N$ and some constant $C$ depending only on $n, N$ and $\rho$. By the Leibniz rule, using \eqref{e:zetaestimates}, \eqref{a:a0estimate} and \eqref{a:aestimates} we can estimate 
 \begin{align*} 
 \|\frac{\delta^{\sfrac{1}{2}}a \gamma_2(\lambda x\cdot \nu )}{\lambda }\zeta\|_k &\leq C \left(\delta^{\sfrac{1}{2}} \lambda^{k-1}\|a\zeta\|_0 + \delta^{\sfrac{1}{2}} \lambda^{-1}\|a\zeta\|_k\right) \leq  C\delta^{\sfrac{1}{2}}\lambda^{k-1}
 \end{align*}
 for $k=0,\ldots,N$. Lastly, by \eqref{e:Testimate} and \eqref{a:westimates},
 \[ \|\delta T w\|_k \leq C\delta\left( \|w\|_k + \|T\|_k\|w\|_0\right) \leq C\delta\lambda^{k-1} \]
 for $k=0,\ldots, N-j+1$, where $C$ depends on $C_0, n, N$ and $\rho$. Combining these three estimates yields \eqref{e:stepestimates} in view of $0<\delta<1$.

 To show \eqref{e:steperror} we compute (suppressing the argument $\lambda x\cdot \nu$ in $\gamma_1,\gamma_2$)
  \begin{align*}
      Dv &= Du + \delta a^2 \gamma_1' T\nu \otimes \nu + \delta^{\sfrac{1}{2}}a\gamma_2'\zeta\otimes \nu \\ 
       & \quad\quad \,\,\,\,\,  +\delta T Dw + \delta (DT) w\\ 
       & \quad\quad \,\,\,\,\, + \delta \frac{\gamma_1}{\lambda }D(a^2 T\nu) + \delta^{\sfrac{1}{2}} \frac{\gamma_2}{\lambda} D(a\zeta)\\
       & =: Du + A_1 +A_2 \\ 
       & \quad\quad \,\,\,\,\,\,+B_1 + B_2 \\
       & \quad\quad \,\,\,\,\,\, +E_1+ E_2\,.
  \end{align*}
  Now write 
  \begin{equation}\label{e:induced} 
  Dv^t Dv = Du^t Du + 2\,\sym(Du^t (A+B+E)) + (A+B+E)^t(A+B+E)\,,\end{equation}
  where we abbreviated $A:= A_1+A_2$, $B:=B_1+B_2$ and $E:= E_1+E_2$. 
  
  For the second term on the right hand side of \eqref{e:induced}, note that 
  \[ 2 \,\sym \left (Du^t(A+B_1)\right) = 2\delta a^2 \gamma_1'\nu\otimes \nu + 2\delta\, \sym\left( Dw\right)\, \]
  by $Du^t\zeta=0$ and $Du^tT=\mathrm{Id}$. The next term in $Du^t (A+B+E)$ we denote by \[R_1 := 2\sym(Du^t B_2 ),\]
  which satisfies
  \begin{align}\label{e:R_1}
       \|R_1\|_k&\leq  C\delta(\|T\|_{k+1}\|w\|_0+\|T\|_1\|w\|_k+\|T\|_1\|w\|_0\|u\|_{k+1}) \nonumber\\
       &\leq C\delta^{\sfrac32} \lambda^{k}, \text{ for } k=0,\ldots, N-j+1,
  \end{align}
  where we used \eqref{a:uinvertibility}, \eqref{e:Testimate} and \eqref{a:westimates} and the constant $C$ depends on $\rho, N$ and $C_0$. 
   Proceeding with the next terms, we define $M_1$ and $M_2$ through 
  \[\delta \gamma_i M_i := 2\,\sym(Du^t E_i)\]
  for $i=1,2$, i.e.,
  \begin{align}
      M_1 &= \frac{2}{\lambda}\sym\left(Du^t D(a^2 T\nu)\right)\,,\label{d:M_1}\\
      M_2 &= \frac{2}{\delta^{\sfrac{1}{2}}\lambda} \,\sym\left( Du^t D(a\zeta)\right)= \frac{2a}{\delta^{\sfrac{1}{2}}\lambda} \,\sym\left( Du^t D\zeta\right)\,,\label{d:M_2}
  \end{align}
  so that we can write
  \begin{equation}\label{e:dudv}
      2\,\sym (Du^t (A+B+E))= 2\delta a^2 \gamma_1'\nu\otimes \nu + 2\delta\,\sym\left( Dw\right) +\delta(\gamma_1 M_1+\gamma_2 M_2)+R_1\,.
  \end{equation}

Now consider the last term in \eqref{e:induced} and observe that the only terms in this product which are not of order $O(\delta^{\sfrac{3}{2}}) $ are 
  \[ A_2^t A_2, \quad 2\,\sym ( A_2^t E_2 ), \quad E_2^t E_2 \,.\]
Computing the first term gives 
\[ A_2^t A_2 = \delta a^2 (\gamma_2')^2\nu\otimes \nu\,.\]
For the second one, we can define $M_3$ and $\gamma_3$ through
\[\delta \gamma_3 M_3 = 2\, \sym(A_2^t E_2)\,,\]
that is,
\begin{equation}\label{d:M_3}
      M_3= 2\frac{a}{\lambda}\sym\left(\nu\otimes \zeta D(a\zeta)\right)\,,~~ \gamma_3 = (\gamma_2')^2.
  \end{equation}
For the last term we observe that due to $|\zeta|=1$ we have 
\[ E_2^tE_2 = \delta \frac{\gamma_2^2}{\lambda^2}\nabla a\otimes \nabla a + \delta \frac{\gamma_2^2}{\lambda^2}D\zeta^t D\zeta\,.\]
Notice that the first quantity is not of the form $\gamma M$ for any $\gamma \in C^\infty(\S^1)$ with mean zero, indeed the mean of $\gamma_2^2$ is $2\pi$. We set
  \begin{equation}\label{d:M_4}
      M_4 =  \frac{1}{\lambda^2}\nabla a\otimes \nabla a\,,
  \end{equation} 
so that
\[\delta \frac{\gamma_2^2}{\lambda^2}\nabla a\otimes \nabla a=\delta\gamma_4M_4+\delta\frac{2\pi}{\lambda^2}\nabla a\otimes \nabla a.\]
We therefore obtain 
  \[ E_2^tE_2 = \delta \gamma_4 M_4 + \delta\frac{2\pi}{\lambda^2}\nabla a\otimes \nabla a +\delta\frac{\gamma_2^2}{\lambda^2}D\zeta^tD\zeta\,.\]
  Besides, by \eqref{a:uestimates}, the last term of the above formula can be controlled by
\begin{equation}\label{e:dzeta2}
   \|\delta\frac{\gamma_2^2}{\lambda^2}D\zeta^t D\zeta\|_k \leq C\delta^2 \frac{\mu^2}{\lambda^2}\lambda^k < C\delta^{\sfrac{3}{2}}\lambda^k\,\text{ for } k=0,\ldots, N,
\end{equation}
so that it is small enough to be put into $R.$ 
Finally, we can set
 \begin{align}
      R_2 :=& (A+B+E)^t(A+B+E) -A_2^tA_2 -2\, \sym(A_2^t E_2) \nonumber\\
      &~~~- \left(E_2^tE_2 - \delta \frac{\gamma_2^2}{\lambda^2}D\zeta^t D\zeta\right)\nonumber\\
      =& (A+B+E)^t(A+B+E) -\delta a^2(\gamma_2')^2\nu\otimes\nu \nonumber\\
      &~~~-\delta (\gamma_3 M_3+\gamma_4 M_4) -\delta\frac{2\pi}{\lambda^2}\nabla a\otimes \nabla a\,.\label{d:R2}
 \end{align}
From \eqref{e:induced}, \eqref{e:dudv}, and \eqref{d:R2}, we have
  \begin{align*}
 Dv^t Dv^t =& Du^t Du^t + 2\delta a^2 \gamma_1'\nu\otimes \nu + 2\delta\, \sym\left( Dw\right) +\delta(\gamma_1 M_1+\gamma_2 M_2) +R_1\\
 & + \delta a^2 (\gamma_2')^2\nu\otimes \nu +\delta (\gamma_3 M_3+\gamma_4 M_4)+ \delta\frac{2\pi}{\lambda^2}\nabla a\otimes \nabla a+R_2\,.
  \end{align*}
Upon setting $R:=R_1+R_2,$ from \eqref{e:inclusion} we conclude \eqref{e:steperror} with $M_i, \, i=1,\ldots, 4$ defined in \eqref{d:M_1}, \eqref{d:M_2}, \eqref{d:M_3} and \eqref{d:M_4} respectively.

It remains to check that the estimates \eqref{e:steperrorestimatesM}--\eqref{e:steperrorestimatesR} hold. From the assumptions \eqref{a:uestimates}, \eqref{a:a0estimate},  \eqref{a:aestimates}, \eqref{e:Testimate} and the Leibniz rule \eqref{e:Leibniz} it follows directly 
\begin{align*}
    \|M_1\|_k \leq& \frac{C}{\lambda} \big(\|a\|_{k+1}+\|T\|_{k+1}+(\|u\|_{k+1}+\|T\|_k)\|a\|_1\\
    &~~~~~+(\|u\|_{k+1}+\|a\|_k)\|T\|_1\big)\\
    \leq&C\frac{\mu}{\lambda}\mu^k,\quad k=0,\ldots, N.
\end{align*}
Similarly,
\begin{align*}
&\|M_2\|_k + \|M_3\|_k \leq C\frac{\mu}{\lambda}\mu^k,\,\\
& \|M_4\|_k \leq C\frac{\mu}{\lambda^2}\mu^{k+1} \leq C\frac{\mu}{\lambda}\mu^k\,,
\end{align*} 
for a constant $C$ only depending on $n$, $N$ and $\rho$,  which shows \eqref{e:steperrorestimatesM}.
For \eqref{e:steperrorestimatesR} we compute  
\begin{align*}
     &\|A_1\|_k \leq C \delta \lambda^k\,,\quad \|A_2\|_k \leq C\delta^{\sfrac{1}{2}}\lambda^k, \\ 
     & \|E_1\|_k \leq C\delta\frac{\mu}{\lambda}\lambda^k\,,\quad 
     \|E_2\|_k\leq C\delta^{\sfrac{1}{2}}\frac{\mu}{\lambda}\lambda^k\,, 
\end{align*}
 for $ k=0,\ldots, N\,$ and
\begin{align*}
    \|B_1\|_k \leq C\delta \lambda^k\,,\qquad \|B_2\|_k \leq C\delta \lambda^{k-1} \text{ for } k=0,\ldots, N-j\,,
\end{align*}
using the assumptions and the Leibniz rule. Then
the Leibniz rule, $\delta<1$ and $1\leq \mu\leq \lambda$  imply
\begin{align*}
    &\|A_1^t(A+B+E)\|_k \leq C\delta^{\sfrac{3}{2}}\lambda^k \text{ for }k=0,\ldots, N-j\,,\\
  &\|A_2^t(A_1+B+E_1)\|_k \leq C\delta^{\sfrac{3}{2}}\lambda^k \text{ for } k=0,\ldots, N-j\,,\\
  &  \|B^t(A+B+E)\|_k \leq C\delta^{\sfrac{3}{2}}\lambda^k\,\text{ for } k=0,\ldots, N-j\,,\\
  &\| E_1^t(A+B+E)\|_k \leq C\delta^{\sfrac{3}{2}}\lambda^k\,,\text{ for }  k=0,\ldots, N-j\,.
\end{align*} 
Note that 
 \begin{align*}
     R =& (A+B+E)^t(A+B+E) -A_2^tA_2 -2\, \sym(A_2^t E_2) - E_2^tE_2 + \delta \frac{\gamma_2^2}{\lambda^2}D\zeta^t D\zeta\\
     =& 2\sym\big(A_1^t(A+B+E)+A_2^t(A_1+B+E_1)+B^t(A+B+E)\big)\\
     &+2\sym\big(E_1^t(A+B+E)\big)+ \delta \frac{\gamma_2^2}{\lambda^2}D\zeta^t D\zeta\,.
 \end{align*}
Thus, combining \eqref{e:R_1} with \eqref{e:dzeta2}, we get \eqref{e:steperrorestimatesR}. This finishes the proof.
\end{proof}

Setting $w=0$ in \eqref{d:perturbation} gives the the following ``standard" step proposition, which will be applied to the final $n_*-n$ perturbations. 

\begin{proposition}[Ordinary Step]\label{p:usualstep}
Let $V\subset \R^n$ be an open, bounded, simply connected set.
Fix $N\in \N$.  Let $u\in C^{N+2}(\bar V,\R^{n+1})$ be an immersion and $a\in C^{N+1}(\bar V)$ such that
\begin{align}
   & \frac{1}{\rho} \mathrm{Id}\leq Du^t Du \leq \rho \mathrm{Id}\,,\label{a:uestmate-11}\\
    &\|u\|_{k+1}\leq \delta^{\sfrac{1}{2}} \mu^k \text{ for } k=1,\ldots,N\,, \label{a:ustimates-12}\\
    &\|a\|_0\leq \rho,\label{a:a0estimate-1},\\
    &\|a\|_k\leq \mu^k \text{ for } k=1,\ldots,N+1\,,\label{a:aestimates-11}
\end{align}
for some  constants $\rho\geq1,\, \, 0<\delta\leq 1$, $\mu\geq \delta^{-\sfrac12}$.  

Then for any  $\lambda \geq \mu$
and  any $\nu \in \S^{n-1}$, there exists a map $v\in C^{N+1}(\bar V, \R^{n+1})$ and 
 $$\mathcal{E}:=Dv^tDv -( Du^tDu + \delta a ^2 \nu \otimes \nu)$$
 such that
\begin{align}
    \|v-u\|_k\leq& C\delta^{\sfrac{1}{2} } \lambda^{k-1} \text{ for } k=0,\ldots, N+1\,,\label{e:v-uk-1}\\
     \|\mathcal{E}\|_0 \leq& C \delta \left(\frac{\mu}{\lambda}+\delta^{\sfrac{1}{2}}\right)\label{e:error-1}
\end{align}
for a constant $C$ only depending on $N, n $ and  $\rho$.
\end{proposition}
\begin{proof}
    The existence of $\zeta$ is guaranteed by \eqref{a:uestmate-11}, comparing again \cite[Lemma 3.5]{CaoIn2024}. Defining $v$ via \eqref{d:perturbation} with $w\equiv0$, we
    have \eqref{e:v-uk-1} and 
    \[Dv^tDv -( Du^tDu + \delta a ^2 \nu \otimes \nu)=R+\delta\frac{2\pi}{\lambda^2}\nabla a\otimes\nabla a+\delta\sum_{i=1}^4\gamma_iM_i.\]
    By \eqref{e:steperrorestimatesM},\eqref{e:steperrorestimatesR} and \eqref{a:aestimates-11}, we obtain
    \begin{align*}
       \|Dv^tDv -( Du^tDu + \delta a ^2 \nu \otimes \nu)\|_0\leq C(\delta^{\sfrac32}+\delta\frac{\mu^2}{\lambda^2}+\delta\frac{\mu}{\lambda}),
    \end{align*}
    which then implies \eqref{e:error-1} by $\mu\leq\lambda$.  
\end{proof}

On the other hand, if $\nu\cdot e_1 \neq 0$ we can choose a suitable $w$ in the ansatz \eqref{d:perturbation} with the help of Proposition \ref{p:ibp}, which improves the structure and size of the error.

\begin{proposition}[Sharper Step]\label{p:ibpstep}  
Let $V\subset \R^n$ be an open, bounded, simply connected set.
Fix $N\in \N$.  Let $u\in C^{N+2}(\bar V,\R^{n+1})$  be an immersion and $a\in C^{N+1}(\bar V)$ such that
\begin{align}
   & \frac{1}{\rho} \mathrm{Id}\leq Du^t Du \leq \rho \mathrm{Id}\,,\label{a:uestmate-21}\\
    &\|u\|_{k+1}\leq \delta^{\sfrac{1}{2}} \mu^k \text{ for } k=1,\ldots,N\,, \label{a:ustimates-22}\\
    &\|a\|_0\leq\rho, \label{a:a0estimate-2}\\
    &\|a\|_k\leq \mu^k \text{ for } k=0,\ldots,N+1\,,\label{a:aestimates-2}
\end{align}
for some  constants $\rho\geq1,\, 0<\delta\leq 1$, $\mu\geq \delta^{-\sfrac12}.$ 

Then  for any $\lambda \geq \mu$, any $\nu\in \S^{n-1} $ with $\nu \cdot e_1 \neq 0$, and any natural number $I\leq N$,  there exists a map $v\in C^{N-I+1}(\bar V, \R^{n+1})$ and error terms $\mathcal{E}\in C^{N-I}$ and $\mathcal{F}\in C^{N-I+1}(\bar V,\Sym)$ such that 
\[Dv^tDv=Du^tDu + \delta a ^2 \nu \otimes \nu + \delta \frac{2\pi}{\lambda^2}\nabla a\otimes \nabla a +\mathcal{E}+\mathcal{F}\,,\]
and 
\begin{align}
    &\|v-u\|_k\leq C\delta^{\sfrac{1}{2} } \lambda^{k-1} \text{ for } k=0,\ldots, N-I+1\,,\label{e:v-uk-2}\\
    &\|\mathcal{E}\|_0\leq
    C\delta \left( \left(\frac{\mu}{\lambda}\right)^{I+1} + \delta^{\sfrac{1}{2}}\right)\,,\label{e:error-2}\\
    &\mathcal{F}\in \mathrm{span}\{\nu_j\otimes \nu_j : j=n+1,\ldots,n_*\}\,\, \label{e:Fspan},\\
    & \|\mathcal{F}\|_k \leq C\delta \frac{\mu}{\lambda} \lambda^k \qquad \text{ for } k=0,\ldots, N-I+1\,\label{e:Festi}
\end{align}
for a constant $C$ only depending on $N, n $ and $\rho$.
\end{proposition}
\begin{proof}
As in the proof of Proposition \ref{p:usualstep}, the existence of $\zeta$ satisfying \eqref{a:normal} and the estimates \eqref{e:zetaestimates} follows from \eqref{a:uestmate-21}. We define $v$ through \eqref{d:perturbation} with
$w$ to be chosen. From Lemma \ref{l:step} it follows that $M_i$ in \eqref{e:steperror} satisfy \eqref{assM} with $K= \bar C\frac{\mu}{\lambda}$ for $i=1,\ldots, 4.$  Since $\nu \cdot e_1 \neq 0,$  we can apply Proposition \ref{p:ibp} to find 
\[ w_i^I\in C^{N-I+1}(\bar V, \R^n), \quad E^I_i\in C^{N-I}(\bar V, \Sym), \text{ and }F^I_i\in C^{N-I+1}(\bar V, \Sym)\]
such that 
\[ \gamma_i(\lambda x\cdot \nu) M_i = 2\sym(Dw^I_i) + \gamma_i^I(\lambda x\cdot \nu) \left(\frac{\mu}{\lambda}\right)^I E_i^I + F_i^I\,,\]
for $i=1,\ldots, 4$. Moreover, $F_i^I $ satisfies \eqref{e:spanproperty} and $w_i^I, \, E_i^I, \, F^I_i$ satisfy the corresponding estimates \eqref{e:westimate}--\eqref{e:Festimate} with $K=\bar C\frac{\mu}{\lambda}$. We define our desired $w$ by 
     \[ w=- \sum_{i=1}^4 w_i^I\,.\]
We then have $w\in C^{N-I+1}(\bar V, \R^n)$ with 
\[\|w\|_k \leq CK\lambda^{k-1}\leq C_0\lambda^{k-1} \text{ for } k=0,\ldots, N-I+1\,,\]
so that $w$ satisfies the assumptions \eqref{a:westimates} in Lemma \ref{l:step} for a constant $C_0$ only depending on $n,N$ and $\gamma$. Set
\[F=\sum_{i=1}^4F_i^I\,,\text{ and }\mathcal{F}=\delta F,\]
then $\mathcal{F}\in C^{N-I+1}(\bar V, \Sym)$  satisfies \eqref{e:Fspan} and \eqref{e:Festi}. Finally, from \eqref{e:steperrorestimatesR} and \eqref{e:Eestimate} we find
\begin{align*}
    &\|Dv^tDv-(Du^tDu + \delta a ^2 \nu \otimes \nu + \delta \frac{2\pi}{\lambda^2}\nabla a\otimes \nabla a +\mathcal{F})\|_0\\
    = &\|\delta\sum_{i=1}^4 \gamma_i^I \left(\frac{\mu}{\lambda}\right)^I E_i^I+R\|_0\leq C\delta\left(\frac{\mu}{\lambda}\right)^{I+1}+C\delta^{\sfrac32},
\end{align*}
which then gives \eqref{e:error-2}. 
\end{proof}

\subsection{Adding a metric deficit--Stage Proposition}\label{s:stage}
With the above step propositions we can prove the main iteration proposition. Recall that $h_0 $ is the constant matrix defined in \eqref{d:h_0}. 
\begin{proposition}[Stage]\label{p:stage} Let $U\subset \R^n$ be an open, bounded, simply connected domain. 
Assume $u\in C^2(\bar U,\R^{n+1})$ is an immersion such that 
    \begin{align}
        &\|g- Du^tDu - \delta h_0 \|_0 \leq r\delta\,, \label{a:metricdeficit}\\
        &\|u \|_2 \leq \delta^{\sfrac{1}{2}}\lambda\label{a:C2}
    \end{align}
for some constants $r>0, \, 0<\delta\leq1$ and $\lambda\geq \delta^{-\sfrac{1}{2}}$. 
For any $J\in\mathbb{N},$ there exist positive constants $\delta_*(n, J, g),\, r_*(n, J, g), \,c_*(n, J, g)$ such that the following holds.

If $r\leq r_*$ and $\delta\leq \delta_*, $ then for any simply connected $V \Subset U$, and any 
\[\hat\delta<\frac{r\delta}{|h_0|},  \quad \Lambda\geq c_*\]
there exists an immersion $v\in C^2(\overline{V}, \R^{n+1})$ such that 
\begin{align}
        &\|g- Dv^tDv - \hat\delta h_0 \|_0 \leq C\delta\left(\Lambda^{-J}+\delta^{\sfrac{1}{2}}\right) + \lambda^{-2}\,, \label{e:metricdeficit}\\
        &\|v-u\|_k \leq  C \delta^{\sfrac{1}{2}}\lambda^{k-1}\,, \text{ for }k=0,1\,, \label{e:C1}\\
        &\|v\|_2 \leq C\delta^{\sfrac{1}{2}}\lambda \eta^{-1}\Lambda^{J(n_*-n)+n}\,,\label{e:C2}
\end{align}
 where $\eta=\min\{1, \mathrm{dist}(V, \partial U)\}$ and the constant $C$ only depends on $n, g, J$.
\end{proposition}
Observe that, up to lower order terms, the metric deficit decreases by a factor of $K=\Lambda^J$, while the $C^2$ norm increases by a of factor $\Lambda^{J(n_*-n)+n}$. For sufficiently large $J$, this heuristically leads to a growth of the $C^2$ norm by a factor $K^{n_*-n}$, which in turn determines the exponent \eqref{e:exponent}, as discussed in Section \ref{s:mainideas}. The ``loss" of domain from $U $ to $V\Subset U$ in the above proposition is due to a mollification step required to manage the loss of derivatives.
\begin{proof}
As outlined above, we begin by decomposing the metric deficit $g-Du^tDu-\delta h_0$ into $n_* $ primitive metrics using Lemma \ref{l:babykaellen}. We then iteratively add the first $n$ primitive metrics using Proposition \ref{p:ibpstep}, followed by the remaining ones with Proposition \ref{p:usualstep}. However, to manage the loss of derivatives, we first mollify both the immersion and the metric deficit at a suitable length scale.

{\it Step 1. Mollification.}  
We mollify $u$ and $g$ at length-scale 
\begin{equation}\label{d:ell}
    \ell  = \frac{\eta}{\hat C\lambda}\,,
\end{equation}
where $\hat C \geq 1$ is a constant to be chosen below.
Denote $$u_\ell = u\ast \varphi_\ell \text{ and } g_\ell = g\ast \varphi_\ell.$$ Observe that both maps are defined on $\bar V$ since $\hat C\lambda\geq 1$. Moreover, from Lemma \ref{l:mollification} we get the estimates 
\begin{align}
& \|u-u_\ell\|_k \leq C\|u\|_2\ell^{2-k}\leq C\delta^{\sfrac12}\lambda^{k-1}\text{ for } k=0,1,2\,,\label{e:uell-1}\\
& \|u_\ell\|_k \leq C \|u\|_2 \ell^{2-k}  \leq C\delta^{\sfrac{1}{2}}\ell^{1-k}\text{ for } k=2, 3,\ldots, N_*+2\label{e:uell-2}
\end{align} 
for some constants $C$ only depending on $N_*$ and $n$, where we set
\begin{equation}\label{e:Nstar}
    N_*=J(n+1)+n_*\,.
\end{equation}
In particular, the constant only depends on $n, J$.
Moreover, if $\hat C \geq \sqrt{\|g\|_2}$, then 
\begin{equation}\label{e:gmollified}
    \|g-g_\ell\|_0 \leq C\ell^2 \|g\|_2\leq \frac{1}{\lambda^2}\,.
\end{equation}

\smallskip

{\it Step 2. Decomposition. } 
Define the (rescaled) deficit 
\begin{equation}\label{d:H} 
H= \frac{g_\ell - Du_\ell^tDu_\ell}{\delta} - \frac{\hat \delta }{\delta }h_0\,.
\end{equation}
Then $H\in C^\infty(\bar V,\Sym)$ by construction. Note that
\[H - h_0 = \frac{1}{\delta}((g- Du^tDu - \delta h_0)\ast \varphi_\ell)
+ \frac{1}{\delta }((Du^tDu)\ast \varphi_\ell - Du_\ell^tDu_\ell) - \frac{\hat\delta}{\delta }h_0\,.
\]
Hence,
\begin{align*}
    \|H-h_0\|_0 &\leq \frac{1}{\delta} \|(g- Du^tDu - \delta h_0)\ast \varphi_\ell\|_0  + \frac{\hat \delta}{\delta }|h_0| \\
    &\qquad + \frac{1}{\delta} \|(Du^tDu)\ast \varphi_\ell - Du_\ell^tDu_\ell\|_0\\
    &\leq r + r + C \frac{\ell^2 \|u\|_2^2}{\delta}\leq 2 r +  \frac{C}{\hat C^2}\leq \frac{r_2}{2}
\end{align*}
if $r<r_*\leq\frac{r_2}{16}$ and $\hat C$ is large enough depending on $n $ and the constant $r_2=r_2(n,N_*)=r_2(n, J)$ in Lemma \ref{l:babykaellen}. So $\hat C$ depends on $r_2, n, J, g$, i.e., $n,J$ and $g$. 
Similarly,
\begin{align}\label{e:Hestim}
\|H-h_0\|_k&\leq \frac{1}{\delta} \|(g- Du^tDu - \delta h_0)\ast \varphi_\ell\|_k + \frac{1}{\delta} \|(Du^tDu)\ast \varphi_\ell - Du_\ell^tDu_\ell\|_k\nonumber\\
    &\leq C \ell^{-k} +C \frac{\ell^{2-k} \|u\|_2^2}{\delta} \leq C\ell^{-k}
    \end{align}
for any $k=0,\ldots, N_*+1$ and a constant $C$ only depending on $n,N_*$. We first take $\Lambda\geq \frac{2}{r_2}$ and set
\begin{equation}\label{d:lambda}
\lambda_0 = C\ell^{-1}\,, \quad \lambda_i = \Lambda^i \lambda_0 \text{ for } i =1, \ldots, n\,,
\end{equation}
where $C$ is the constant in the above \eqref{e:Hestim}.
In particular, \[\frac{\lambda_0}{\lambda_1} \leq \frac{r_2}{2}.\]
Hence, we can apply the decomposition Lemma \ref{l:babykaellen} with $N=N_*$ and $j=J$ to obtain 
\[a:=a^J = (a^J_1,\ldots, a^J_{n_*}) \in C^{N_*+1-J}(\bar V, \R^{n_*})\text{ and } \mathcal{E}_0:= E^J \in C^{N_*-J}(\bar V, \Sym)\]
satisfying 
\begin{align}
    H= &\sum_{i=1}^{n_*} a_i^2 \nu_i\otimes \nu_i + \sum_{l=1}^n \frac{2\pi}{\lambda_l^2}\nabla a_l\otimes \nabla a_l + \mathcal{E}_0\,,\label{e:Hdecom}\\
    \|a\|_k \leq& C \lambda_0^k \text{ for } k=0,\ldots, N_*-J+1,\label{e:aJestimate}\\
    \|\mathcal{E}_0\|_k \leq& C \Lambda^{-2J-2}\lambda_0^k \text{ for } k=0,\ldots, N_*-J, \label{e:E0estimate}
\end{align}
for a constant $C$ only depending on $N_*, \,J,\, n$, i.e. $J, \,n$.

\smallskip

{\it Step 3. First $n$ perturbations.} Observe that from \eqref{a:metricdeficit} it follows 
$$g-(1+r)\delta \mathrm{Id} \leq Du^t Du \leq g.
$$
Consequently, if $r \leq r_* $ and $\delta_*$ is small enough depending on $g$, we can find $\rho$ only depending on $n, J, g$  such that 
\[\frac{1}{\rho}\textrm{Id}\leq Du^tDu\leq \rho\textrm{Id}, \quad \|a\|_0\leq \rho, \]
where we also used  \eqref{e:aJestimate} for $k=0$. 
Combining with \eqref{e:uell-1}, choosing $\delta$ smaller if necessary, we further have
\begin{equation}
    \label{e:uellimmersion}
    \frac{1}{\rho_0}\textrm{Id}\leq Du_\ell^tDu_\ell\leq \rho_0\textrm{Id}, \quad \|a\|_0\leq\rho_0 \text{ with }\rho_0=2\rho.
\end{equation}
We also define 
\begin{equation}\label{e:rhoidef}
\rho_i=2\rho_{i-1}  \text{ for }i=1,\cdots, n_*\,.
\end{equation}

Set $$u_0 = u_\ell.$$ 
Recall that the basis $\{\nu_i\otimes\nu_i\}$ is ordered such that $\nu_i\cdot e_1\neq0$ for $i=1,\ldots, n$.  
Utilizing \eqref{e:uell-1}, \eqref{e:uell-2} and  \eqref{e:aJestimate}, for $i=1, \ldots, n$,  starting with $u_0$ and \eqref{e:uellimmersion}, after taking $\Lambda$ large, we iteratively apply Proposition \ref{p:ibpstep} with
\[N=N_*-J-(i-1)(J+1),\quad I=J,\]
and
\[u=u_{i-1}, \quad \rho=\rho_{i-1},\quad \mu=C\lambda_{i-1}, \quad \lambda=\lambda_i, \]
to get 
\[u_i\in C^{N_*-J-(J+1)i+2}(\bar V, \R^{n+1}),\quad \mathcal{E}_i\in C^{N_*-J-(J+1)i+1}(\bar V, \Sym)\]
such that
\begin{align}
    &\frac{1}{\rho_i}\textrm{Id}\leq Du_i^tDu_i\leq \rho_i\textrm{Id},\label{e:uiimmersion}\\
    &Du_i^tDu_i = Du^t_{i-1}Du_{i-1} + \delta a_i^2 \nu_i\otimes \nu_i + \delta\frac{2\pi}{\lambda_i^2} \nabla a_i \otimes \nabla a_i +  \mathcal{E}_i+\mathcal{F}_i,\label{e:error-i}\\
    &\|u_i-u_{i-1}\|_k \leq C\delta^{\sfrac{1}{2}} \lambda_i^{k-1} \text{ for } k= 0,\ldots, N_*-J-(J+1)i+2\,,\label{e:uidifference}\\
    &\|\mathcal{E}_i\|_0 \leq C \delta(\Lambda^{-J-1}+\delta^{\sfrac12})\,,\label{e:Eiestimate}
\end{align}
and 
\[ \mathcal{F}_i \in C^{N_*-J-(J+1)i+2}(\bar V, \Sym)\]
satisfying 
\begin{align}
    &\mathcal{F}_i\in \mathrm{span}\{\nu_j\otimes \nu_j : j=n+1,\ldots,n_*\}\,, \label{e:Fispan}\\
  &\|\mathcal{F}_i\|_k \leq C\delta\Lambda^{-1}\lambda_i^k \text{ for } k=0,\ldots, N_*-J-(J+1)i+2 \label{e:Fiestimate}
       \end{align}
for a constant $C$ only depending on $J, \,n,\, g$.  We only need to verify \eqref{e:uiimmersion}, since \eqref{e:error-i}-\eqref{e:Fiestimate} follow directly from Proposition \ref{p:ibpstep}. In fact, by induction, with \eqref{e:uiimmersion} for $i-1$, employing \eqref{e:uidifference} for $k=1$, we compute
\[Du_i^tDu_i-Du_{i-1}^tDu_{i-1}=Du_i^t(Du_i-Du_{i-1})+(Du_i-Du_{i-1})^tDu_{i-1}\]
to deduce \eqref{e:uiimmersion} up to choosing $\delta$ smaller. Thus, in this step we have constructed
 $$u_n\in C^{N_*-J-(J+1)n+2}(\bar V, \R^{n+1})$$
 such that 
 \begin{equation}\label{e:firststeps}
     \begin{split} Du_n^tDu_n &= Du^t_0Du_0+\delta  \sum_{i=1}^n a_i^2 \nu_i\otimes \nu_i + \delta\sum_{i=1}^n\frac{2\pi}{\lambda_i^2} \nabla a_i \otimes \nabla a_i + \mathcal{E}+\mathcal{F}\,,
     \end{split}\end{equation}
by \eqref{e:error-i}, where we abbreviated
\[ \mathcal{E}:=\sum_{i=1}^n \mathcal{E}_i, \quad \mathcal{F}:=\sum_{i=1}^n\mathcal{F}_i\,. 
\]
By \eqref{e:uidifference}, \eqref{e:uell-1} and \eqref{e:uell-2}, we have
\begin{equation}
    \label{e:unestimate}
    \|u_n\|_{k+1}\leq\sum_{i=1}^n\|u_i-u_{i-1}\|_{k+1}+\|u_0\|_{k+1}\leq C\delta^{\sfrac12}\lambda_n^k
\end{equation}
for $k=1, \ldots, N_*-J-(J+1)n+2.$ From \eqref{e:Eiestimate} and \eqref{e:Fiestimate}, we also have
\begin{align}
\|\mathcal{E}\|_0&\leq \delta(\Lambda^{-J-1}+\delta^{\sfrac12}) \,,\label{e:finalerror1} \\
\|\mathcal{F}\|_k& \leq C\delta\Lambda^{-1}\sum_{i=1}^n\lambda_i^k \leq C \delta\Lambda^{-1}\lambda_n^k \label{e:Festimate-1n}
\end{align}
 for $ k=0,\ldots, N_*-J-(J+1)n+2\,.$ Moreover, by \eqref{e:Fispan}, we get
 \begin{equation}
     \label{e:Fdecomposition}
     \mathcal{F}\in \mathrm{span}\{\nu_j\otimes \nu_j: j=n+1,\ldots, n_*\}\,.
 \end{equation}
 \smallskip

 {\it Step 4. Remaining $n_*-n$ perturbations.} Since $\nu_j\cdot e_1 = 0$ for any $j=n+1,\ldots, n_*$, Proposition \ref{p:ibpstep} cannot be applied. We now add the remaining $n_*-n $ primitive metrics in the decomposition of  $H$ by applying Proposition \ref{p:usualstep} iteratively. Additionally, since \eqref{e:Fdecomposition} holds, we can simultaneously cancel the large error term $\mathcal F $.  By \eqref{e:Fdecomposition} and Lemma \ref{l:matrixdecomposition} we write 
\begin{equation}
    \label{e:Fdecom}
    \mathcal{F}= \sum_{j=n+1}^{n_*} L_j(\mathcal{F}) \nu_j\otimes \nu_j\,,
\end{equation} 
where $L_j$ are the linear maps from Lemma \ref{l:matrixdecomposition}. It then follows
\[\|L_j(\mathcal{F})\|_0\leq C \|\mathcal{F}\|_0 \leq C\delta\Lambda^{-1}.\]
 Define the adjusted amplitudes $b_j$ by 
\begin{equation}\label{d:b}
b_j = \sqrt{a_j^2-\delta^{-1}L_j(\mathcal{F})}\,,
\end{equation}
for $j=n+1,\ldots,n_*$.  Since $(a_j)^2\geq r_2^2$ by Lemma \ref{l:babykaellen}, it follows that $b_j$ is well defined if  we choose $\Lambda$ so large that $$C\Lambda^{-1}\leq\frac14r_2^2\,,$$
and then $b_j, j=n+1,\ldots, n_*$ are well-defined. Furthermore, by  Lemma \ref{l:composition} and estimates \eqref{e:ajestimate} and \eqref{e:Festimate-1n}, we get
\begin{equation}\label{e:bestimates}
  \|b_j\|_k\leq C \lambda_n^k \text{ for } k=0,\ldots, N_*-J-(J+1)n+2\,,
\end{equation}
For $i=n+1,\ldots, n_*$ we now fix the frequencies 
\begin{equation}\label{d:lastfrequencies}
     \lambda_{i}= \Lambda^J\lambda_{i-1} = \Lambda^{(i-n)J+n}\lambda_0\,.
\end{equation}
Note that from Step 3, we have
\begin{equation}
    \label{e:unimmersion}
    \frac1{\rho_n}\textrm{Id}\leq Du_n^tDu_n\leq\rho_n\textrm{Id}.
\end{equation}
Hence, with \eqref{e:unimmersion}, \eqref{e:unestimate} and \eqref{e:bestimates}, starting with $u_n$, for $i=n+1,\ldots, n_*$, we can inductively apply Proposition \ref{p:usualstep} with \[N=N_*-J-n(J+1)-(i-(n+1)),\] and
\[u=u_{i-1}, \quad \rho:=\rho_{i-1}~(\text{defined in \eqref{e:rhoidef}}),\quad \mu=C\lambda_{i-1}, \quad \lambda=\lambda_i ~(\text{defined in \eqref{d:lastfrequencies}}), \]
to get 
\[u_i\in C^{N_*-J-n(J+1)n-(i-n)+2}(\bar V, \R^{n+1}),\quad \mathcal{E}_i\in C^{N_*-J-(J+1)n-(i-n)+1}(\bar V, \Sym)\]
such that
 \begin{align}
 &\frac{1}{\rho_i}\textrm{Id}\leq Du_i^tDu_i\leq\rho_i\textrm{Id},\label{e:uniimmersion}\\
 &\|u_i-u_{i-1}\|_k \leq C\delta^{\sfrac{1}{2}} \lambda_i^{k-1} \text{ for } k= 0,\ldots,N_*-J-(J+1)n-(i-n)+2\,,\label{e:uni-diff}\\
      &Du_i^tDu_i = Du^t_{i-1}Du_{i-1} +\delta (a_i^2-\delta^{-1}L_i(\mathcal{F}))\nu_i\otimes \nu_i +  \mathcal{E}_i\,,\label{e:error-ni}\\
      &\|\mathcal{E}_i\|_0\leq C\delta \left(\Lambda^{-J}+\delta^{\sfrac{1}{2}}\right)\,,
      \label{e:Eniestimate}
 \end{align} 
 where the constant $C$ only depends on $J$ and $n$. Note that \eqref{e:uniimmersion} follows from the same argument as that in Step 3 and the other conclusions \eqref{e:uni-diff}-\eqref{e:Eniestimate} follow from Proposition \ref{p:usualstep} and the choice of $\lambda_i$ in \eqref{d:lastfrequencies}. Therefore, in this step we get a map
 \[u_{n_*}\in C^{2}(\bar V, \R^{n+1})\]
 since 
\[N_*-J-(J+1)n-(n_*-n)+2=N_*-J(n+1)-n_*+2=2\]
 by \eqref{e:Nstar}.  Moreover, $u_{n_*}$ satisfies 
 \begin{equation}\label{e:laststeps}
 Du_{n_*}^tDu_{n_*} = Du^t_{n}Du_{n} +\sum_{i=n+1}^{n_*}\delta (a_i^2-\delta^{-1}L_i(\mathcal{F}))\nu_i\otimes \nu_i +  \mathbb{E}\,,
 \end{equation}
by summing \eqref{e:error-ni} from $n+1$ to $n_*$,
where we wrote
\[\mathbb{E}=\sum_{i=n+1}^{n_*}\mathcal{E}_i\,.\]
By \eqref{e:uni-diff} and \eqref{e:unestimate}, we get
\begin{equation}
    \label{e:unstar}
    \|u_{n_*}\|_{k+1}\leq\sum_{i=n+1}^{n_*}\|u_i-u_{i-1}\|_{k+1}+\|u_n\|_{k+1}\leq C\max\{1, \delta^{\sfrac12}\lambda_{n_*}^k\}
\end{equation}
for $k=0, 1$. From \eqref{e:Eniestimate}, we also have
\begin{equation}
    \label{e:finalerror2}
    \|\mathbb{E}\|_0\leq \sum_{i=n+1}^{n_*}\|\mathcal{E}_i\|_0\leq C\delta \left(\Lambda^{-J}+\delta^{\sfrac{1}{2}}\right).
\end{equation}
\smallskip

{\it Step 5. Conclusion. } Set \[v:=u_{n_*}\in C^{2}(\bar V, \R^{n+1}).\] Next we show $v$ is our desired immersion. From \eqref{e:laststeps} and \eqref{e:firststeps}, we obtain
 \begin{align*} 
 Dv^tDv =& Du_n^tDu_n + \sum_{i=n+1}^{n_*}\delta (a_i^2-\delta^{-1}L_i(\mathcal{F}))\nu_i\otimes \nu_i +  \mathbb{E}\\
  =&  Du_0^tDu_0 +\delta  \sum_{i=1}^n a_i^2 \nu_i\otimes \nu_i + \delta\sum_{i=1}^n\frac{2\pi}{\lambda_i^2} \nabla a_i \otimes \nabla a_i + \mathcal{E} +  \mathcal{F}\\
 &+\sum_{i=n+1}^{n_*}\delta a_i^2\nu_i\otimes \nu_i-\sum_{i=n+1}^{n_*}L_i(\mathcal{F}))\nu_i\otimes \nu_i +  \mathbb{E}\,.
\end{align*}
By \eqref{e:Fdecom}, \eqref{e:Hdecom} and  \eqref{d:H}, we have
\begin{align*}
 Dv^tDv =&  Du_0^tDu_0 +\delta  \sum_{i=1}^{n_*} a_i^2 \nu_i\otimes \nu_i + \delta\sum_{i=1}^n\frac{2\pi}{\lambda_i^2} \nabla a_i \otimes \nabla a_i + \mathbb{E} + \mathcal{E} \\
  =&  Du_0^tDu_0 + +\delta \left(H -\mathcal{E}_0\right) +\mathbb{E} + \mathcal{E}\\
  =& g_\ell-\hat \delta h_0 +\delta\mathcal{E}_0+\mathbb{E} + \mathcal{E}\,.
 \end{align*}
 Combining \eqref{e:gmollified}, \eqref{e:E0estimate}, \eqref{e:finalerror1} and \eqref{e:finalerror2} yields 
\begin{align*}
    \|g-Dv^tDv -\hat \delta h_0 \|_0&\leq \|g-g_\ell\|_0 + \delta\|\mathcal{E}_0\|_0 + \|\mathcal{E}\|_0+\|\mathbb{E}\|_0 \\
&\leq C\|g\|_2\ell^2 + C\delta\left(\Lambda^{-2J-2} +\Lambda^{-(J+1)}+ \Lambda^{-J}+\delta^{\sfrac{1}{2}}\right) \\
&\leq \lambda^{-2}+C\delta\left( \Lambda^{-J}+\delta^{\sfrac{1}{2}}\right)
\end{align*}
using $\Lambda\geq 1 $. This shows  \eqref{e:metricdeficit}. 

It remains to show \eqref{e:C1} and \eqref{e:C2}. The estimate \eqref{e:C2} follows from \eqref{e:unstar} and
\[\lambda_{n_*}=\lambda_0\Lambda^{J(n_*-n)+n}=\hat C \eta^{-1} \lambda\Lambda^{J(n_*-n)+n},\]
where we used \eqref{d:lambda}  and \eqref{d:lastfrequencies}.
For \eqref{e:C1} we combine \eqref{e:uidifference} and \eqref{e:uni-diff} to find for $k=0,1$
\begin{align*}
    \|v-u\|_k\leq& \|u-u_\ell\|_k+\sum_{i=1}^{n_*} \|u_{i}-u_{i-1}\|_k \\
    \leq &C\ell^{2-k} \|u\|_2+ C\delta^{\sfrac{1}{2}}\sum_{i=1}^{n_*} \lambda_i^{k-1} \\
    \leq &C\delta^{\sfrac{1}{2}}\lambda^{k-1}\,,
\end{align*}
where we used the definition of $\ell$ in \eqref{d:ell} and $\lambda_i\geq \lambda_0=\ell^{-1}\geq \lambda$ for any $i$.  The conclusion that $v$ is an immersion follows from  \eqref{e:uniimmersion}. This completes the proof. 
\end{proof}

\section{Proof of Theorem \ref{t:main}}\label{s:pfoftmain}
Given the short immersion $\underline u$  we want to iteratively apply Proposition \ref{p:stage} to generate a sequence $\{u_q\}$ of short immersions converging in $C^{1,\alpha}$ to an isometric immersion $u$ in any $C^0$ neighbourhood of $\underline u$. In order to apply Proposition \ref{p:stage} we first construct an short immersion satisfying the conditions in Proposition \ref{p:stage}. This is achieved by a classical Nash-Kuiper iteration. We then iterate Proposition \ref{p:stage} to generate $\{u_q\}$ such that $\{Du_q^tDu_q\}$ converge to the target metric $g$. The proof is divided into three subsections.


\subsection{Initial short immersion}
 Because of the mollification step at the beginning of each stage proposition, the maps $\{u_q\}$ will be defined on smaller and smaller domains. In order to end up with an isometry on $\Omega$ we therefore need to extend both the metric $g$ and the short map $\underline u$.
 Extend $g$ and $\underline u$ to $\R^n$ such that
\[ \|\underline u\|_{C^1(\R^n)} \leq C\|\underline u\|_{C^1(\bar \Omega)}\,, \quad \| g\|_{C^2(\R^n)} \leq C\|g\|_{C^2(\bar \Omega)}\,,\]
for a constant only depending on $n, \,\Omega$. Such an extension procedure is well-known (see e.g. \cite{Whitney}).  By mollification and scaling  we can also assume that $\underline u$ is smooth and strictly short for $g$ on $\bar \Omega$.
By continuity and convexity we can find an open and simply connected set $V_0 \Subset \R^n$ depending on $\underline u$ and $g$, and $0<\underline\delta<\delta_*$ with $\delta_*$ from Proposition \ref{p:stage} such that 
\begin{align}
    &\Omega \Subset V_0,\nonumber\\
    & g  - D\underline u^tD\underline u> 2\underline\delta h_0 \text{ on } \bar V_0\,,\label{e:short1}\\
    &\underline u \text{ is an immersion on } \bar V_0\nonumber
\end{align}



By \eqref{e:short1}, for any $0<\delta_0<\underline\delta$, we have
\begin{align*}
    g-D\underline u^tD\underline u - \delta_0 h_0 \geq\underline\delta h_0.
\end{align*}
We can therefore decompose 
\[ g-D\underline u^tD\underline u - \delta_0 h_0  = \sum_{i=1}^N a_i^2 \eta_i\otimes \eta_i \] 
for some finite number $N\in \N$ of unit vectors $\eta_i \in \S^{n-1}$ and $a_i \in C^\infty(\bar V_0)$ (see for example \cite[Step 1 in Section 7]{DIS}, \cite{Laszlo}). Then by \cite[Proposition 3.2]{CaoSze2022} we obtain for any $\mu\geq C_0(\underline u, g)$ a smooth immersion $u_0$ and metric error $\mathcal{E}_0$ such that
\[Du_0^tDu_0=D\underline u^tD\underline u+\sum_{i=1}^N a_i^2 \eta_i\otimes \eta_i+\mathcal{E}_0\]
with
\begin{align}\label{e:u0-estimate}
    \|u_0-\underline u\|_0\leq \frac{C_1}{\mu},\quad \|u_0\|_2\leq C_1\mu^{N},\quad \|\mathcal{E}_0\|_0\leq \frac{C_1}{\mu}\,,
\end{align}
for a constant $C_1$ depending on $\underline u,g$ and $N$, but not on $\delta_0$. Thus we have
\begin{equation}
    \label{e:u0-deficit}
    g-Du_0^tDu_0-\delta_0 h_0=-\mathcal{E}_0\,.
\end{equation}
We take
\[\delta_0:=a^{-\tau}\,,\, \quad \lambda_0:=a^{(N+1)\tau}\]
with some constant $1>\tau>0$ and large constant $a$ to be fixed below. For $r\leq 1$ to be fixed below, we choose
\begin{equation}\label{d:mu}
     \mu:= \max\{C_1r^{-1}a^{\tau}, 2C_1\epsilon^{-1}, C_0(\underline u, g)\}=C_1r^{-1}a^{\tau},
\end{equation}
provided that $a$ is large. In particular, with this choice we have
\begin{equation}\label{e:delta0}
\frac{C_1}{\mu}\leq\min\{\frac\epsilon2,\, r\delta_0\}.
\end{equation}
Taking $a$ larger depending on $\underline u, g, r, \tau, \epsilon$ and $N$ we can ensure
\begin{equation}
    \label{e:lambda0}
    \delta_0^{\sfrac12}\lambda_0=a^{(N+\frac12)\tau}\geq C_1^{N+1} r^{-N}a^{N\tau}=C_1\mu^N.
\end{equation}
Hence from \eqref{e:u0-deficit}, \eqref{e:delta0} and \eqref{e:lambda0}, we get the initial short immersion $u_0\in C^\infty(\bar V_0, \R^{n+1})$ such that
\begin{equation}\label{e:u0-induction}
\begin{split}
    &\|u_0-\underline u\|_0\leq\frac{\epsilon}{2}, \quad \|u_0\|_2\leq\delta_0^{\sfrac{1}{2}}\lambda_0,\\
    &\|g-Du_0^tDu_0-\delta_0 h_0\|_0\leq r\delta_0.
\end{split} 
\end{equation}

\subsection{Iteration} Fix any desired H\"older exponent 
\[ \alpha <\frac{1}{n^2-n+1}= \frac{1}{1+2(n_*-n)}\,\]
and choose $J\in \N$ large such that 
\begin{equation}\label{e:Jexponent}
    \beta:=\frac{J}{J(1+2(n_*-n))+4n}\in(\alpha,\,\frac{1}{1+2(n_*-n)}) \,.
\end{equation}  
Now we choose a sequence $\{V_q\}_{q\in \N}$ of open and simply connected domains such that $\Omega \Subset  V_{q+1} \Subset V_q$ for any $q\geq0$,  $V_q \downarrow \bar\Omega $ and 
\begin{equation}\label{e:distance}
    \mathrm{dist}(V_{q}, \partial V_{q-1}) \geq 2^{-q}\mathrm{dist}(\Omega,\partial V_0) =:\eta_q \text{ for }q\geq1\,.
\end{equation} 
Define
\begin{equation}
    \label{e:deltalambda}
    \delta_{q}=\delta_0a^{1-b^q},\quad \lambda_{q}=\lambda_0 a^{\frac{1}{2\beta}(b^q-1)}
\end{equation}
for some $1<b<1+\frac\tau2<\frac32$ to be chosen below. Set
\begin{equation}
    \label{e:Lambda}
    \Lambda_q=K(\delta_{q}\delta_{q+1}^{-1})^{\frac1J}
\end{equation}
with constant $K$ to be fixed.  Then we can take $a>a_0(\underline u, g, r, n, J, b, N, \mathrm{dist}(\Omega,\partial V_0))$ sufficiently large such that the claim below holds.  

{\bf Claim:} for any $q\geq0,$ we are able to apply Proposition \ref{p:stage} iteratively to construct a sequence of immersions $\{u_q\}_{q\in \N}$ 
such that 
\begin{align}
    &u_q\in C^2(\bar V_q,\R^{n+1})\,, \label{e:qass0}\\
    &\|g- Du_q^tDu_q- \delta_q h_0\|_0\leq r\delta_q\,, \label{e:qass1}\\
    &\|u_q\|_2\leq \delta_q^{\sfrac{1}{2}}\lambda_q\,, \label{e:qass3}\\
    &\|u_{q+1}- u_{q} \|_k\leq C\delta_{q}^{\sfrac12}\lambda_q^{-k}, \text{ for } k=0, 1,\label{e:qass4} 
\end{align}
where the  constant $C$ depends on $n, g, \underline u.$

Indeed, by \eqref{e:u0-induction}, $u_0$ has been constructed if we choose $r\leq r_*(n,J,g)$ in the definition of $\mu$ in \eqref{d:mu}. Here, $r_*(n,J,g)$ is the constant in Proposition \ref{p:stage} with $J$ defined as in \eqref{e:Jexponent}. Now   assume $u_q$ has already been constructed. Taking $a$ large we have
 \[\delta_{q+1}\leq \frac{r\delta_q}{|h_0|}\,.\] 
Upon choosing the constant $K$ in \eqref{e:Lambda} so large that $K>c_*$ with $c_*$ from Proposition \ref{p:stage} we can apply  Proposition \ref{p:stage}
with
 \[ U=V_q,\, V=V_{q+1},\, u=u_q\,,\, \delta = \delta_q\,,\, \hat\delta = \delta_{q+1}\,,\,\lambda=\lambda_{q}, \,\eta = \eta_{q+1},\, \Lambda=\Lambda_q\] 
 to get $v\in C^2(\bar V_q, \R^{n+1})$. Set
 \[u_{q+1}:=v\in C^2(\bar V_q, \R^{n+1}).\]
 Then \eqref{e:qass0} holds for $q+1$. To show \eqref{e:qass1} for $q+1$,  using  \eqref{e:metricdeficit}, \eqref{e:deltalambda}, \eqref{e:Lambda} and $b<1+\frac\tau2$, we get 
 \begin{align*}
     \|g- Du_{q+1}^tDu_{q+1}- \delta_{q+1} h_0\|_0\leq& C\delta_q(\Lambda_q^{-J}+\delta_q^{\sfrac12})+\lambda_q^{-2}\\
     \leq & C\delta_{q+1}\big(K^{-J}+\delta_{q+1}^{-1}\delta_q^{\sfrac32}+\delta_{q+1}^{-1}\lambda_q^{-2}\big)\\
     \leq &C\delta_{q+1}\big(K^{-J}+a^{b-1-\frac\tau2}+a^{-2N\tau-\tau-1+b}\big)\\
     \leq& r\delta_{q+1}\,,
 \end{align*}
after taking $K\geq 3C$ and $a$ larger. By \eqref{e:C2}, we have
 \begin{align*}
     \|u_{q+1}\|_2\leq& C\delta_{q}^{\sfrac12}\lambda_q\eta_{q+1}^{-1}\Lambda_q^{J(n_*-n)+n}=\delta_{q+1}^{\sfrac12}\lambda_{q+1}A
 \end{align*}
with
\begin{align*}
    A:=& C\frac{\delta_{q}^{\sfrac12}\lambda_q}{\delta_{q+1}^{\sfrac12}\lambda_{q+1}}\eta_{q+1}^{-1}\Lambda_q^{J(n_*-n)+n}\\
     \leq&Ca^{(\frac12-\frac1{2\beta})(b-1)b^q+(b-1)b^q(n_*-n+\frac{n}{J})}K^{n_*-n+\frac{n}{J}}2^{q+1}\mathrm{dist}(\Omega,\partial V_0)^{-1}\,.
\end{align*}
By \eqref{e:Jexponent}, we have
\[\frac1{2\beta}-\frac12=n_*-n+\frac{2n}{J}\]
and then
\[A\leq Ca^{-\frac{n}{J}(b-1)b^q}K^{n_*-n+\frac{n}{J}}2^q\mathrm{dist}(\Omega,\partial V_0)^{-1}\leq 1\]
after taking $a\geq a_0(\underline u, g, r, n, J, b, \mathrm{dist}(\Omega,\partial V_0))$ large. Thus we get \eqref{e:qass3}. Moreover, \eqref{e:qass4} follows directly from \eqref{e:C1}. The claim is then proved.

\subsection{Convergence} Finally, we show that the sequence constructed in the previous section converges in $C^{1,\alpha}(\bar\Omega, \R^{n+1}) $ to an isometric immersion on $\bar \Omega$. By interpolation and $\delta_{q}^{\sfrac12}\lambda_{q}\leq \delta_{q+1}^{\sfrac12}\lambda_{q+1}$ for any $q$ it holds
\begin{align*} \|u_{q+1}-u_{q}\|_{1,\alpha} &\leq C_\alpha \|u_{q+1}-u_{q}\|_1^{1-\alpha}\|u_{q+1}-u_{q}\|_2^\alpha\\
&\leq C\delta_q^{\frac12(1-\alpha)}(\delta_{q+1}^{\sfrac12}\lambda_{q+1})^\alpha\\
&=C\delta_0^{\sfrac12}\lambda_0^\alpha a^{\frac12-\frac{\alpha}{2\beta}}a^{\frac12b^q(\frac{\alpha}{\beta}{b}-1-(b-1)\alpha)}\,.
\end{align*}
Since $\alpha<\beta$, we can take 
\[1<b<\min\left\{1+\frac\tau 2,\, \frac{\beta-\alpha\beta}{\alpha-\alpha\beta}\right\}=1+\min\left\{\frac\tau2,\,\frac{\beta-\alpha}{\alpha-\alpha\beta}\right\}\]
and also $a$ large such that 
\[\|u_{q+1}-u_{q}\|_{1,\alpha}\leq \left(\delta_0^{\sfrac12}\lambda_0^\alpha a^{\frac12-\frac{\alpha}{2\beta}}\right) 2^{-q-1},
\]
and then $\{u_q\}$ is a Cauchy sequence in $C^{1,\alpha}(\bar\Omega, \R^{n+1})$. Let $u\in C^{1,\alpha}(\bar\Omega, \R^{n+1})$ be the limit. From \eqref{e:qass1} and $\delta_q\to0$ as $q\to\infty,$ we have
\[g-Du^tDu=\lim_{q\to\infty}(g-Du_q^tDu_q)=0.\]
Thus $u$ is an isometric immersion for $g$ on $\bar\Omega.$ Furthermore, by \eqref{e:qass4} for $k=0$, we have
\begin{align*}
    \|u-u_0\|_0\leq& \sum_{q=0}^\infty\|u_{q+1}-u_q\|_0\leq C\sum_{q=0}^\infty\delta_q^{\sfrac12}\lambda_q^{-1}\\
    \leq& C\delta_0^{\sfrac12}\lambda_0^{-1}a^{\frac12+\frac1{2\beta}}\sum_{q=0}^\infty a^{-b^q(\frac12+\frac1{2\beta})}\\
    \leq& C a^{-\frac12\tau-(N+1)\tau}a^{\frac12+\frac1{2\beta}}\sum_{q=0}^\infty a^{-(q+1)(\frac12+\frac1{2\beta})}\\
    \leq& Ca^{-\frac12\tau-(N+1)\tau}\leq\frac\epsilon2,
\end{align*}
by taking $a$ sufficiently large.  The above inequality together with \eqref{e:u0-induction} implies
\[\|u-\underline u\|_0\leq\|u-u_0\|_0+\|u_0-\underline u\|\leq\epsilon\,.\]
Therefore, the proof is complete.
\smallskip

\begin{appendix}
\section{H\"older spaces and its properties}
Let $\beta=(\beta_1, \ldots, \beta_n)\in\R^n$ be any multi-index and $|\beta|=\beta_1+\beta_2+\cdots+\beta_n.$ For any function $f:\R^n\supset\Omega\rightarrow\R$ and $m\in\mathbb{N}$, we define its  supreme norms as follows.
\begin{equation*}
\|f\|_0=\sup_{x\in\Omega}f(x), ~~\|f\|_m=\sum_{j=0}^m\max_{|\beta|=j}\|\partial^{{\beta}} f\|_0,
\end{equation*}
Its H\"older semi-norms are defined as
\begin{align*}
[f]_{\alpha}&=\sup_{x\neq y}\frac{|f(x)-f(y)|}{|x-y|^\alpha},\\
[f]_{m+\alpha}&=\max_{|\beta|=m}\sup_{x\neq y}\frac{ |\partial^{{\beta}} f(x)-\partial^{{\beta}} f(y)|}{|x-y|^\alpha},
\end{align*}
for any $ 0<\alpha\leq1.$ Here  $\partial^{{\beta}}=\partial_{x_1}^{\beta_1}\partial_{x_2}^{\beta_2}\cdots\partial_{x_n}^{\beta_n}.$
Then the H\"older norms are given as
$$\|f\|_{m+\alpha}=\|f\|_m+[f]_{m+\alpha}. $$
Let $C^{m, \alpha}(\overline \Omega)$ be the H\"older space  defined on $\Omega$ such that $f\in C^{m, \alpha}(\overline\Omega)$ if $\|f\|_{m+\alpha}$ is finite. Interpolation inequality of H\"older norms reads
$$[f]_r\leq C\|f\|_0^{1-\frac{r}{s}}[f]_s^{\frac{r}{s}}$$
for $s>r\geq0.$  
If $f_1, f_2\in C^{m, \alpha}(\overline\Omega)$ for any $m\in\mathbb{N}$ and $\alpha\in[0, 1)$, then we have
\begin{equation}\label{e:Leibniz}
\|f_1f_2\|_{m+\alpha}\leq C(m, \alpha)(\|f_1\|_0\|f_2\|_{m+\alpha}+\|f_1\|_{m+\alpha}\|f_2\|_0).
\end{equation}
Moreover, we have the following lemma for compositions of H\"older continuous maps (see e.g. \cite{DIS}).
\begin{lemma}
    \label{l:composition}
    Let $u:\R^n\to\Omega$ and $F:\Omega\to\R$ be two smooth functions, with $\Omega\subset\R^N$. Then, for every natural number $m>0$ , there is a constant $C$
 (depending only on $m,N,n$) such that
 \begin{align*}
     [F\circ u]_m\leq& C([F]_1[u]_m+\|DF\|_{m-1}[u]_0^{m-1}[u]_m);\\
     [F\circ u]_m\leq& C([F]_1[u]_m+\|DF\|_{m-1}[u]_1^{m-1}).
 \end{align*}
\end{lemma}
We also recall some estimates about the regularization of H\"older functions (see e.g. \cite{CDS,DIS}).  Let $\varphi_l$ be a  mollification kernel at length-scale $l>0$ and $*$ denote convolution. 
\begin{lemma}\label{l:mollification}
For any $r, s\geq0,$ and $0<\alpha\leq1,$ we have
\begin{align*}
&[f*\varphi_l]_{r+s}\leq Cl^{-s}[h]_r,\\
&\|f-f*\varphi_l\|_r\leq Cl^{1-r}[h]_1, \text{ if } 0\leq r\leq1,\\
&\|(f_1f_2)*\varphi_l-(f_1*\varphi_l)(f_2*\varphi_l)\|_r\leq Cl^{2\alpha-r}\|f_1\|_\alpha\|f_2\|_\alpha,
\end{align*}
with constant $C$ depending only on $n, s, r, \alpha$ and $\varphi.$
\end{lemma}
\end{appendix}

\section*{Acknowledgements}
Wentao Cao's research was supported by Grant Agreement No. 12471224 of the National Natural Science Foundation of China.
%

\bibliographystyle{plain}

\end{document}